\documentclass{svmult}

\usepackage{makeidx}
\usepackage{graphicx}
\usepackage{multicol}
\usepackage[bottom]{footmisc}
\usepackage{amsmath,amssymb,amsfonts}

\newcommand{\R}{{\mathbb{R}}} 
\newcommand{\C}{{\mathbb{C}}}
\newcommand{\F}{{\mathbb{F}}} 
\newcommand{\OO}{{\mathbb{O}}} 
\newcommand{\HH}{{\mathbb{H}}}

\newcommand{\CC}{{\mathcal{C}}} 
\newcommand{\MI}{{\mathcal{MI}}} 
\newcommand{\MA}{{\mathcal{A}}} 
\newcommand{\MB}{{\mathcal{B}}} 
\newcommand{\MM}{{\mathcal{M}}} 
 
\newcommand{\Gl}{\operatorname{Gl}}
\newcommand{\RR}{{\mathcal{R}}} 
 
\newcommand{\Scap}{\operatorname{Cap}}

\newcommand{\Hom}{\operatorname{Hom}}
\newcommand{\diag}{\operatorname{diag}}

\newcommand{\Aut}{\operatorname{Aut}}
\newcommand{\End}{\operatorname{End}}
\newcommand{\trace}{\operatorname{trace}}
\newcommand{\supp}{\operatorname{supp}}

\newcommand{\Harm}{\operatorname{Harm}}
\newcommand{\1}{\operatorname{\bf 1}}

\newcommand{\Stab}{\operatorname{Stab}}

\newcommand{\Id}{\operatorname{Id}}

\newcommand{\Pold}{\operatorname{Pol}_{\leq d}}

\newcommand{\prodtrace}[2]{\langle #1, #2 \rangle}

\newcommand{\kr}{\text{cr}}

\newcommand{\kfrac}[2]{\mbox{$\frac{#1}{#2}$}}
\newcommand{\oN}{{\mathbb{N}}}
\newcommand{\oR}{{\mathbb{R}}}

\newtheorem{defi}{Definition}[section]
\newtheorem{mydefinition}[defi]{Definition}
\newtheorem{myproposition}[defi]{Proposition}
\newtheorem{mytheorem}[defi]{Theorem}
\newtheorem{myremark}[defi]{Remark}
\newtheorem{mycorollary}[defi]{Corollary}
\newtheorem{mylemma}[defi]{Lemma}
\newtheorem{myexample}[defi]{Example}

\newcommand*{\transp}{\mathsf{T}}
\newcommand*{\q}{\mathbf{q}}

\newcommand{\Pow}{{\mathcal{P}}}

\makeindex

\begin{document}

\title*{Invariant semidefinite programs}

\author{Christine Bachoc\inst{1}\and
Dion C.~Gijswijt\inst{2}\and
Alexander Schrijver\inst{3}\and
Frank Vallentin\inst{4}}

\authorrunning{C.~Bachoc, D.C.~Gijswijt, A.~Schrijver, F.~Vallentin} 

\institute{Laboratoire A2X, Universit\'e Bordeaux I, 351,
cours de la Li\-b\'e\-ration, 33405 Talence, France
\texttt{bachoc@math.u-bordeaux1.fr}
\and
CWI and Department of Mathematics, Leiden University;
Centrum voor Wiskunde en Informatica (CWI),
Sciencepark 123, 1098 XG Amsterdam, The Netherlands
\texttt{dion.gijswijt@gmail.com}
\and
CWI and Department of Mathematics, University of Amsterdam;
Centrum voor Wiskunde en Informatica (CWI),
Sciencepark 123, 1098 XG Amsterdam, The Netherlands
\texttt{lex@cwi.nl}
\and 
Delft Institute of Applied Mathematics, 
Technical University of Delft, P.O. Box 5031, 2600 GA Delft, The Netherlands
\texttt{f.vallentin@tudelft.nl}}

\maketitle

\section{Introduction}

In the last years many results in the area of semidefinite programming were obtained for invariant semidefinite programs --- semidefinite programs which have symmetries. This was done for a variety of problems and applications. The purpose of this handbook chapter is to give the reader the necessary background for dealing with semidefinite programs which have symmetry. Here the focus is on the basic theory and on representative examples. We do not aim at completeness of the presentation.

In all applications the underlying principles are similar: one simplifies the original semidefinite program which is invariant under a group action by applying an algebra isomorphism mapping a ``large'' matrix algebra to a ``small'' matrix algebra. Then it is sufficient to solve the semidefinite program using the smaller matrices.

We start this chapter by developing the general framework in the introduction where we give a step-by-step procedure for simplifying semidefinite programs: Especially  Step 2 (first version), Step 2 (second version), and Step $1 \frac{1}{2}$ will be relevant in the later discussion. Both versions of Step 2 are based on the main structure theorem for matrix $*$-algebras. Step $1 \frac{1}{2}$ is based on the regular $*$-representation.

In Section~\ref{sec:matrix star} we give a proof of the main structure theorem for matrix $*$-algebras and we present the regular $*$-representation. Strictly speaking the framework of matrix $*$-algebras is slightly too general for the applications we have in mind. However, working with matrix $*$-algebras does not cause much extra work and it also gives a numerical algorithm for finding an explicit algebra isomorphism. Section~\ref{sec:matrix star} is mainly concerned with finite dimensional invariant semidefinite programs. In Section~\ref{Representation theory} we show how one can extend this to special classes of infinite dimensional invariant semidefinite programs, namely those which arise from permutation actions of compact groups. We focus on this case because of space limitations and because it suffices for our examples. This section is connected to Step 2 (second version) of the introduction.

The later sections contain examples coming from different areas: In Section~\ref{Block codes} we consider finding upper bounds for finite error-correcting codes and in Section~\ref{Crossing numbers} we give lower bounds for the crossing number of complete bipartite graphs. Both applications are based on the methods explained in Section~\ref{sec:matrix star} (Step 2 (first version) and Step $1 \frac{1}{2}$ in the introduction). In Section~\ref{Spherical codes} we use Step 2 (second version) for finding upper bounds for spherical codes and other geometric packing problems on the sphere. For this application the background in Section~\ref{Representation theory} is relevant. Section~\ref{Block codes} and Section~\ref{Spherical codes} both use the theta number of Lov\'asz for finding upper bounds for the independence number of highly symmetric graphs. In Section~\ref{Sums of squares} we show how one can exploit symmetry in polynomial optimization: We give particular sum of squares representations of polynomials which have symmetry. 

This list of applications is not complete, and many more applications can be found in the literature. In the last Section~\ref{More applications} we give literature pointers to more applications.

\subsection{Complex semidefinite programs}

In order to present the theory as simple as possible we work with complex semidefinite programs. We give the necessary definitions. A complex matrix $X \in \C^{n \times n}$ is \emph{Hermitian} if $X = X^*$, where $X^*$ is the \emph{conjugate transpose} of $X$, i.e.\ $X_{ij} = \overline{X_{ji}}$. It is called \emph{positive semidefinite}, we write $X \succeq 0$, if for all (column) vectors $(\alpha_1, \ldots, \alpha_n) \in \C^n$ we have
\begin{equation*}
\sum_{i=1}^n \sum_{j = 1}^n \alpha_i X_{ij} \overline{\alpha_j} \geq 0.
\end{equation*}
The space of complex matrices is equipped with a complex inner product, the trace product $\prodtrace{X}{Y} = \trace(Y^*X)$, which is linear in the first entry. The inner product of two Hermitian matrices is always real.

\begin{mydefinition}
A \emph{(complex) semidefinite program} is an optimization problem
of the form
\begin{equation}
\label{sdp standard}
\max\{\prodtrace{X}{C} : X \succeq 0,\prodtrace{X}{A_1} = b_1, \ldots, \prodtrace{X}{A_m} = b_m\},
\end{equation}
where $A_1, \ldots, A_m \in \C^{n \times n}$, and $C \in \C^{n \times n}$ are given
Hermitian matrices, $(b_1, \ldots, b_m) \in \R^m$ is a given vector and $X \in
\C^{n \times n}$ is a variable Hermitian matrix. 
\end{mydefinition}

A Hermitian matrix $X \in \C^{n \times n}$ is called a \emph{feasible solution} of \eqref{sdp standard} if it is positive semidefinite and fulfills all $m$~linear constraints. It is called an \emph{optimal solution} if it is feasible and if for every feasible solutions $Y$ we have $\prodtrace{X}{C} \geq \prodtrace{Y}{C}$. 

There is an easy reduction from complex semidefinite programs to semidefinite programs involving real matrices only, as noticed by Goemans and Williamson~\cite{GoemansWilliamson}. A complex matrix $X \in \C^{n \times n}$ defines a real matrix 
\begin{equation*}
\begin{pmatrix}
\Re(X) & -\Im(X)\\
\Im(X) & \Re(X)
\end{pmatrix}
\in \R^{2n \times 2n},
\end{equation*}
where $\Re(X) \in \R^{n \times n}$ and $\Im(X) \in \R^{n \times n}$ are the real and imaginary parts of~$X$. Then the properties of being Hermitian and being complex positive semidefinite translate into being symmetric and being real positive semidefinite: We have for all real vectors $\alpha = (\alpha_1, \ldots, \alpha_{2n}) \in \R^{2n}$:
\begin{equation*}
\alpha^{\sf T}
\begin{pmatrix}
\Re(X) & -\Im(X)\\
\Im(X) & \Re(X)
\end{pmatrix}
\alpha \geq 0.
\end{equation*}

On the other hand, complex semidefinite programs fit into the framework of conic programming (see e.g.\ Nemirovski \cite{Nemirovski}). Here one uses the cone of positive semidefinite Hermitian matrices instead of the cone of real positive semidefinite matrices. There are implementations available, SeDuMi (Sturm \cite{Sturm}) for instance, which can deal with complex semidefinite programs directly.

\subsection{Semidefinite programs invariant under a group action}
\label{ssec:sdp invariant}

Now we present the basic framework for simplifying a complex semidefinite program which has symmetry, i.e.\ which is invariant under the action of a group.

Let us fix some notation first. Let $G$ be a finite group. Let $\pi : G \to U_n(\C)$ be a \emph{unitary representation} of $G$, that is, a group homomorphism from the group $G$ to the group of unitary matrices $U_n(\C)$. The group $G$ is acting on the set of Hermitian matrices by
\begin{equation*}
(g,A) \mapsto \pi(g)A\pi(g)^*.
\end{equation*}
In general, a (left) \emph{action} of a group $G$ on a set $M$ is a map
\begin{equation*}
G \times M \to M, \quad (g, x) \mapsto gx,
\end{equation*}
that satisfies the following properties: We have $1x = x$ for all $x \in M$ where $1$ denotes the neutral element of $G$. Furthermore, $(g_1 g_2)(x) = g_1(g_2 x)$ for all $g_1, g_2 \in G$ and all $x \in M$. 

A matrix $X$ is called \emph{$G$-invariant} if $X = gX$ for all $g \in G$, and we denote the set of all $G$-invariant matrices by $(\C^{n \times n})^G$. We say that the semidefinite program~\eqref{sdp standard} is \emph{$G$-invariant} if for every feasible solution $X$ and for every $g\in G$ the matrix $gX$ is again a feasible solution and if it satisfies $\prodtrace{gX}{C} = \prodtrace{X}{C}$ for all $g \in G$.

One  example, which will receive special attention because of its importance, is the case of a permutation action:
The set of feasible solutions is invariant under simultaneous
permutations of rows and columns. Let $G$ be a finite group which acts
on the index set $[n] = \{1, \ldots, n\}$. So we can see $G$ as a
subgroup of the permutation group on $[n]$. 
To a permutation $\sigma$ on $[n]$ corresponds a matrix permutation
$P_{\sigma}\in U_n(\C)$ defined by 
\begin{equation*}
[P_{\sigma}]_{ij} = 
\left\{
\begin{array}{rl}
1,\;\; & \text{if $i = \sigma(j)$,}\\
0,\;\; & \text{otherwise}.
\end{array}
\right.
\end{equation*}
and the underlying unitary representation is $\pi(\sigma) =
P_{\sigma}$. In this case, the action on matrices $X\in \C^{n\times n}$ is 
\begin{equation*}
\sigma(X)=P_{\sigma}XP_{\sigma}^*,  \text{ where
}\sigma(X)_{ij}=X_{\sigma^{-1}(ij)}=X_{\sigma^{-1}(i),\sigma^{-1}(j)}.
\end{equation*}

\begin{verse}
{\small
\noindent
In the following, we give some background on unitary representations. This part may be skipped at first reading.
Let $G$ be a finite group. A \emph{representation of $G$} is a finite
dimensional complex vector space $V$ together with a homomorphism $\pi: G\to
\Gl(V)$
from $G$ to the group of invertible linear maps on $V$. The space $V$
is also called a \emph{$G$-space} or a \emph{$G$-module} and $\pi$ may
be dropped in notations, replacing 
$\pi(g)v$ by $gv$. In other words, the group $G$ acts on $V$ and this
action has the additional property that $g(\lambda v+\mu w)=\lambda
gv+\mu gw$ for all $(\lambda,\mu)\in \C^2$ and $(v,w)\in V^2$. 
The \emph{dimension} of a representation is the dimension of the
underlying vector space  $V$.

\noindent
If $V$ is endowed with an inner product $\langle v,w\rangle$ which is
invariant under $G$, i.e. satisfies  $\langle gv,gw\rangle=\langle
v,w\rangle$ for all $(v,w)\in V^2$ and $g\in G$, $V$ is called  a
\emph{unitary representation} of $G$. An inner product with this
property always exists on $V$, since it is obtained by taking the
average 
\begin{equation*}
\langle v,w\rangle=\frac{1}{|G|}\sum_{g\in G} \langle gv,gw\rangle_0
\end{equation*}
of an arbitrary inner product $\langle v,w\rangle_0$ on $V$.
So, with an appropriate choice of a basis of $V$, any representation of $V$ is
isomorphic to one of the form $\pi: G \to U_n(\C)$, the form
in which unitary representations of $G$ where defined above. 
}
\end{verse}

\medskip
\noindent\textbf{Step $\mathbf{1}$: Restriction to invariant subspace}
\medskip

Because of the convexity of \eqref{sdp standard}, one can find
an optimal solution of \eqref{sdp standard} in the set of $G$-invariant matrices. In fact, if $X$ is an optimal
solution of \eqref{sdp standard}, so is its \textit{group average}
$\frac{1}{|G|} \sum\limits_{g\in G} gX$. Hence, \eqref{sdp standard} is
equivalent to
\begin{equation}
\label{sdp standard invariant}
\max\{\prodtrace{X}{C} : X \succeq 0, \prodtrace{X}{A_1} = b_1, \ldots, \prodtrace{X}{A_m} = b_m, X \in (\C^{n \times n})^G\}.
\end{equation}

The $G$-invariant matrices intersected with the Hermitian matrices form a vector space. Let $B_1, \ldots, B_N$ be a basis of this space.  Step 1 of simplifying a $G$-invariant semidefinite program is rewriting~\eqref{sdp standard invariant} in terms of this basis.

\medskip

\textit{\textbf{Step 1:} If the semidefinite program~\eqref{sdp standard} is
$G$-invariant, then it is equivalent to}
\begin{equation}
\label{sdp first reduction}
\begin{array}{lcl}
\max\Big\{\prodtrace{X}{C} & : & x_1, \ldots, x_N \in \C,\\
& & X = x_1 B_1 + \cdots + x_N B_N \succeq 0,\\
& & \prodtrace{X}{A_i} = b_i, \; i = 1, \ldots, m\Big\}.
\end{array}
\end{equation}

\medskip

In the case of a permutation action there is a canonical basis of $(\C^{n \times n})^G$ which one can determine by looking at the orbits of the group action on pairs. Then, performing Step 1 using this basis amounts to coupling the variable matrix entries of~$X$. 

The \emph{orbit} of the pair $(i,j) \in [n] \times [n]$ under the group $G$ is given by
\begin{equation*}
O(i,j) = \{(\sigma(i),\sigma(j)): \sigma \in G\}.
\end{equation*}
The set $[n] \times [n]$ decomposes into the orbits $R_1, \ldots,
R_M$ under the action of $G$.  For every $r \in \{1, \ldots, M\}$ we
define the matrix $C_r \in \{0,1\}^{n \times n}$ by $(C_r)_{ij} = 1$ if
$(i,j) \in R_r$ and $(C_r)_{ij} = 0$ otherwise. Then $C_1, \ldots, C_M$
forms a basis of $(\C^{n \times n})^G$, the \emph{canonical basis}. If $(i,j) \in R_r$ we also write
$C_{[i,j]}$ instead of $C_r$. Then, $C_{[j,i]}^{\sf T} = C_{[i,j]}$.

Note here that $C_1, \ldots, C_M$ is a basis of $(\C^{n \times n})^G$. In order to get a basis of the space of $G$-invariant Hermitian matrices we have to consider the orbits of unordered pairs: We get the basis by setting $B_{\{i,j\}} = C_{[i,j]}$ if $(i,j)$ and $(j,i)$ are in the same orbit and $B_{\{i,j\}} = C_{[i,j]} + C_{[j,i]}$ if they are in different orbits.
The matrix entries of $X$ in \eqref{sdp first reduction} are constant on the (unordered) orbits of pairs: $X_{ij} = X_{\sigma(ij)} = X_{\sigma(ji)} = X_{ji}$ for all $(i,j) \in [n] \times [n]$ and $\sigma \in G$.

\medskip
\noindent\textbf{Step $\mathbf{2}$: Reducing the matrix sizes by block diagonalization}
\medskip

The $G$-invariant subspace $(\C^{n \times n})^G$ is closed under matrix multiplication. This can be seen as follows: Let $X, Y \in (\C^{n \times n})^G$ and let $g \in G$, then
\begin{equation*}
g(XY) = \pi(g)XY\pi(g)^* = (\pi(g)X\pi(g)^*)(\pi(g)Y\pi(g)^*) = (gX)(gY) = XY.
\end{equation*}
Moreover, it is also closed under taking the conjugate transpose,
since, for $X\in (\C^{n \times n})^G$,
$\pi(g)X^*\pi(g)^*=(\pi(g)X\pi(g)^*)^*=X^*$. Thus $(\C^{n \times n})^G$
has the structure of a \emph{matrix $*$-algebra}. (However, not all matrix $*$-algebras are coming from group actions.)

In general, a \emph{matrix $*$-algebra} is a set of complex matrices that is closed under addition, scalar multiplication, matrix multiplication, and taking the conjugate transpose. The main structure theorem of matrix $*$-algebras is the following:

\begin{mytheorem}
\label{thm:block diagonal star}
Let $\mathcal{A} \subseteq \C^{n \times n}$ be a matrix $*$-algebra. There are numbers $d$, and $m_1, \ldots, m_d$ so that there is $*$-isomorphism between $\mathcal{A}$ and a direct sum of complete matrix algebras
\begin{equation*}
\varphi : \mathcal{A} \to \bigoplus_{k=1}^d \C^{m_k \times m_k}.
\end{equation*}
\end{mytheorem}

We will give a proof of this theorem in Section~\ref{sub:Blockdiag}
which also will give an algorithm for determining $\varphi$. 
In the case $\mathcal{A}=(\C^{n \times n})^G$, 
the numbers $d$, and $m_1, \ldots, m_d$ have a representation
theoretic interpretation: The numbers are determined by the unitary
representation $\pi : G \to U_n(\C)$, where $d$ is the number of
pairwise non-isomorphic irreducible representations contained in~$\pi$ and where $m_k$ is the multiplicity of the $k$-th  isomorphism class of irreducible representations contained in~$\pi$.

\begin{verse}
\small

\noindent 
In the following we present background in representation theory which is needed for the understanding of the numbers $d$ and $m_1, \ldots, m_d$ in Theorem~\ref{thm:block diagonal star}. Again, this part may be skipped at first reading.

\smallskip

\noindent
A \emph{$G$-homomorphism} $T$ between two $G$-spaces $V$ and $W$ is
a linear map that commutes with the actions of $G$ on $V$
and $W$: for all $g \in G$ and all $v \in V$ we have $T(gv)=gT(v)$.  If $T$ is invertible, then 
it is a \emph{$G$-isomorphism} and $V$ and $W$ are \emph{$G$-isomorphic}.
The set of all $G$-homomorphisms is a linear space, denoted $\Hom^G(V,W)$. 
If $V = W$, then $T$ is said to be a \emph{$G$-endomorphism} and the set $\End^G(V)$
of $G$-endomorphisms of $V$ is moreover an algebra under composition. If $V = \C^n$, and if $G$ acts on $\C^n$ by unitary matrices, then we have $\End^G(V) = (\C^{n \times n})^G = \mathcal{A}$.

\smallskip

\noindent
A representation $\pi : G \to \Gl(V)$ of $G$ (and the corresponding $G$-space~$V$) is \emph{irreducible} if it contains no proper subspace $W$ such that $gW\subset W$ for all $g\in G$, i.e. $V$
contains no \emph{$G$-subspace}. If $V$ contains a proper $G$-subspace
$W$, then also the orthogonal complement $W^{\perp}$ relative to a
$G$-invariant inner product, is a $G$-subspace and $V$ is the
direct sum $W\oplus W^{\perp}$. Inductively, one obtains Maschke's
  theorem: 
\begin{center}
\emph{Every $G$-space is the direct sum of irreducible
$G$-subspaces.}
\end{center}
This decomposition is generally not unique: For example, if $G$ acts trivially
on a vector space $V$ (i.e.\ $gv = v$ for all $g \in G$, $v \in V$) of dimension at least $2$, the irreducible
subspaces are the $1$-dimensional subspaces and $V$ can be decomposed
in many ways. 

\smallskip

\noindent
From now on, we fix a set $\RR=\{R_k : k = 1, \ldots, d\}$ of representatives
of the isomorphism classes of irreducible $G$-subspaces which are direct summands of $V$. 
Starting from an arbitrary decomposition of $V$, we consider, for $k = 1, \ldots, d $, the sum
of the irreducible subspaces which are isomorphic to $R_k$. One can
prove that this $G$-subspace of $V$, is independent of the
decomposition. It is called the {\em isotypic
component} of $V$ associated to $R_k$ and is denoted $\MI_k$. The integer
$m_k$ such that $\MI_k\simeq R_k^{m_k}$ is called the {\em multiplicity} of
$R_k$ in $V$. In other words, we have 
\begin{equation*}
V=\bigoplus_{k = 1}^d \MI_k\quad \text{and }\ 
\MI_k=\bigoplus_{i=1}^{m_k} H_{k,i}
\end{equation*}
where, $H_{k,i}$ is isomorphic to $R_k$, and $m_k\geq 1$. The first
decomposition is orthogonal with respect to an invariant inner product
and is uniquely determined by $V$, while the decomposition of $\MI_k$
is not unique unless $m_k=1$. 

\smallskip

\noindent
\emph{Schur's lemma} is the next crucial ingredient for the description of
$\End^G(V)$: 
\begin{center}
\emph{If $V$ is $G$-irreducible, then $\End^G(V)=\{\lambda\Id : \lambda\in \C\}\simeq \C$.}
\end{center}

\smallskip

\noindent
In the general case, when $V$ is not necessarily $G$-irreducible, 
\begin{equation*}
\End^G(V) \simeq \bigoplus_{k=1}^d \C^{m_k\times m_k}.
\end{equation*}
This result will be derived in the next section, in an algorithmic
way, as a consequence of the more general theory of matrix $*$-algebras.
 \end{verse}

So we consider a $*$-isomorphism $\varphi$ given by Theorem
\ref{thm:block diagonal star} applied to $\mathcal{A}=(\C^{n \times n})^G$:
\begin{equation}\label{block diagonal group}
\varphi :  (\C^{n \times n})^G \to \bigoplus_{k=1}^d \C^{m_k \times m_k}.
\end{equation}

Notice that since $\varphi$ is a $*$-isomorphism between matrix algebras with unity, $\varphi$ preserves also eigenvalues and hence positive semidefiniteness. Indeed, let $X \in (\C^{n \times n})^G$ be $G$-invariant, then also $X - \lambda I$ is $G$-invariant, and $X-\lambda I$ has an inverse if and only if $\varphi(X) - \lambda I$ has a inverse. This means that a test whether a (large) $G$-invariant matrix is positive semidefinite can be reduced to a test whether $d$ (small) matrices are positive semidefinite. Hence, applying $\varphi$ to \eqref{sdp first reduction} gives the second and final step of simplifying \eqref{sdp standard}:

\medskip

\textit{
\textbf{Step 2 (first version):} If the semidefinite program~\eqref{sdp standard} is
$G$-invariant, then it is equivalent to}
\begin{equation}
\label{sdp block first version}
\begin{array}{lcl}
\max\Big\{\prodtrace{X}{C} & : & x_1, \ldots, x_N \in \C,\\
& & X = x_1 B_1 + \cdots + x_N B_N \succeq 0,\\
& & \prodtrace{X}{A_i} = b_i, \; i = 1, \ldots, m,\\
& & x_1 \varphi(B_1) + \cdots + x_N \varphi(B_N) \succeq 0\Big\}.
\end{array}
\end{equation}

Applying $\varphi$ to a $G$-invariant semidefinite program is also called
\emph{block diagonalization}. The advantage of \eqref{sdp block first version} is that instead of dealing with matrices of size $n \times n$ one only has to deal with block diagonal matrices with $d$ block matrices of size $m_1, \ldots, m_d$, respectively. So one reduces the dimension from $n^2$ to $m_1^2 + \cdots + m_d^2$. In the case of a permutation action this sum is also the number of distinct orbits~$M$. In many applications the latter is much smaller than the former. In particular many practical solvers take advantage of the block structure to speed up the numerical calculations.

Instead of working with the $*$-isomorphism $\varphi$ and a basis $B_1, \ldots, B_N$ of the Hermitian $G$-invariant matrices, one can also work with the inverse $\varphi^{-1}$ and the standard basis of $\bigoplus_{k=1}^d \C^{m_k \times m_k}$. This is given by the matrices $E_{k,uv} \in \C^{m_k \times m_k}$ where all entries of $E_{k,uv}$ are zero except the $(u,v)$-entry which equals~$1$. This gives an \emph{explicit parametrization} of the cone of $G$-invariant positive semidefinite matrices.

\medskip

\textit{
\textbf{Step 2 (second version):} If the semidefinite program~\eqref{sdp standard} is
$G$-invariant, then it is equivalent to}
\begin{equation}
\label{sdp block second version}
\begin{split}
\max\Big\{ &\prodtrace{X}{C} : \\
& \quad X = \sum_{k = 1}^d \sum_{u,v=1}^{m_k} x_{k,uv} \varphi^{-1}(E_{k,uv})\\
& \quad x_{k,uv} = \overline{x_{k,vu}}, \; u,v=1, \ldots, m_k,\\
& \quad \begin{pmatrix} x_{k,uv} \end{pmatrix}_{1 \leq u,v \leq m_k} \succeq 0, \; k = 1, \ldots, d,\\
& \quad \prodtrace{X}{A_i} = b_i, \; i = 1, \ldots, m\Big\}.
\end{split}
\end{equation}

Hence, every $G$-invariant positive semidefinite matrix $X$ is of the form
\begin{equation*}
X = \sum_{k =1}^d \sum_{u,v=1}^{m_k} x_{k,uv} \varphi^{-1}(E_{k,uv}),
\end{equation*}
where the $d$ matrices 
\begin{equation*}
X_k = \begin{pmatrix} x_{k,uv} \end{pmatrix}_{1 \leq u,v \leq m_k}, \;\; k = 1, \ldots, d,
\end{equation*}
are positive semidefinite. Define for $(i,j) \in [n] \times [n]$ the matrix $E_k(i,j) \in \C^{m_k \times m_k}$ componentwise by
\begin{equation*}
[E_k(i,j)]_{uv} = \left[\varphi^{-1}(E_{k,uv})\right]_{ij}.
\end{equation*}
By definition we have $E_k(i,j)^* = E_k(j,i)$. Then, in the case of a permutation action, $E_k(i,j) = E_k(\sigma(i),\sigma(j))$ for all $(i,j) \in [n] \times [n]$ and $\sigma \in G$. With this notation one can write the $(i,j)$-entry of $X$ by
\begin{equation}
\label{eq:parameterization}
X_{ij} = \sum_{k = 1}^d \prodtrace{X_k}{E_{k}(i,j)}.
\end{equation}

In summary, finding a block diagonalization of a $G$-invariant
semidefinite program amounts to first identifying a basis of the
Hermitian $G$-invariant matrices and then in finding an explicit
$*$-isomorphism \eqref{block diagonal group} between the algebra of $G$-invariant matrices and the direct sum of complete matrix algebras. In the following sections we will mainly be concerned with different strategies to find such a $*$-isomorphism.

\medskip
\noindent\textbf{Step $\mathbf{1 \frac{1}{2}}$: Reducing the matrix sizes by the regular $*$-representation}
\medskip

In general finding a block diagonalization of a $G$-invariant semidefinite program is a non-trivial task, especially because one has to construct an explicit $*$-isomorphism. In cases where one does not have this one can fall back to a simpler $*$-isomorphism coming from the regular $*$-representation. In general this does not provide the maximum possible simplification. However, for instance in the case of a permutation action, it has the advantage that one can compute it on the level of knowing the orbit structure of the group action only.

For this we consider an orthogonal basis (with respect to the trace inner product $\langle \cdot, \cdot \rangle$) of the $G$-invariant algebra $(\C^{n \times n})^G$. For instance, we can use the canonical basis $C_1, \ldots, C_M$ in the case of a permutation action. By considering the multiplication table of the algebra we define the \emph{multiplication parameters} $p^t_{rs}$, sometimes also called \emph{structural parameters}, by
\begin{equation*}
C_r C_s = \sum_{t = 1}^M p^t_{rs} C_t.
\end{equation*}
In the case of a permutation action the structural parameters can be computed by knowing the structure of orbits (if one chose the canonical basis):
\begin{equation*}
p^t_{rs} = |\{k \in [n] : (i,k) \in R_r, (k,j) \in R_s\}|,
\end{equation*}
where $(i,j) \in R_t$. Here, $p^t_{rs}$ does not depend on the choice of $i$ and $j$. 
The norms $||C_r||=\sqrt{\prodtrace{C_r}{C_r}}$ equal the sizes of the corresponding orbits. We define the matrices $L(C_r)_{st} \in \C^{M \times M}$ by
\begin{equation*}
(L(C_r))_{st} = \frac{\langle C_r C_t, C_s\rangle}{\|C_t\| \; \|C_s\|} = \frac{\|C_s\|}{\|C_t\|} p^s_{rt}.
\end{equation*}

\begin{mytheorem}
\label{thm:regular representation}
Let $\mathcal{L}$ the algebra generated by the matrices $L(C_1), \ldots, L(C_M)$. Then the linear map 
\begin{equation*}
\phi : (\C^{n \times n})^G \to \mathcal{L}, \quad \phi(C_r) = L(C_r), \quad r = 1, \ldots, M,
\end{equation*}
is a $*$-isomorphism.
\end{mytheorem}

We will give a proof of this theorem in Section~\ref{sub:Regularstar}. There we will show that the $*$-isomorphism is the regular $*$-representation of the $G$-invariant algebra associated with the orthonormal basis $C_1/\|C_1\|, \ldots, C_M/\|C_M\|$. Again, since $\phi$ is a $*$-isomorphism between algebras with unity it preserves eigenvalues.  This means that a test whether a $G$-invariant matrix of size $n \times n$ is positive semidefinite can be reduced to testing whether an $M \times M$ matrix is positive semidefinite.

\medskip

\textit{
\textbf{Step $\mathbf{1 \frac{1}{2}:}$} If the semidefinite program~\eqref{sdp standard} is
$G$-invariant, then it is equivalent to \eqref{sdp block first version} where the $*$-isomorphism $\varphi$ is replaced by $\phi$.
}

\section{Matrix $*$-algebras}
\label{sec:matrix star}

In Section~\ref{ssec:sdp invariant} we saw that the process of block diagonalizing a semidefinite program can be naturally done in the framework of matrix $\ast$-algebras using the main structure theorem (Theorem~\ref{thm:block diagonal star}). In this section we prove this main structure theorem. Although we are mainly interested in the case when the matrix $*$-algebra comes from a group, working in the more general framework here, does not cause much extra work. Furthermore the proof of the main structure theorem we give here provides an algorithmic way for finding a block diagonalization. 

We start by giving the basic definitions, examples, and results of matrix $*$-algebras (Section~\ref{sub:definitions}--Section~\ref{sec:hom star}). 
In Section~\ref{sub:Blockdiag} we prove the main structure theorem which gives a very efficient representation of a matrix $*$-algebra $\MA$: We show that $\MA$ is $*$-isomorphic to a direct sum of full matrix $*$-algebras. The corresponding $*$-isomorphism is called a \emph{block diagonalization} of $\MA$. This corresponds to Step $2$ in the introduction. After giving the proof we interpret it in the context of groups and we discuss a numerical algorithm for finding a block diagonalization which is based on the proof. In Section \ref{sub:Regularstar} we consider the regular $*$-representation, which embeds $\mathcal{A}$ into $\C^{M \times M}$, where $M=\dim\MA$. This corresponds to Step $1 \frac{1}{2}$ in the introduction.

\subsection{Definitions and examples}
\label{sub:definitions}

\begin{mydefinition}A \emph{matrix $\ast$-algebra} is a linear subspace $\mathcal{A}\subseteq \C^{n\times n}$ of complex $n\times n$ matrices, that is closed under matrix multiplication and under taking the conjugate transpose. The conjugate transpose of a matrix $A$ is denoted $A^*$.
\end{mydefinition}

Matrix $*$-algebras are finite dimensional $C^*$-algebras and many results here can be extended to a more general setting. For a gentle introduction to $C^*$-algebras we refer to Takesaki \cite{Takesaki}.

Trivial examples of matrix $*$-algebras are the \emph{full matrix algebra} $\C^{n\times n}$ and the \emph{zero algebra} $\{0\}$. From given matrix $*$-algebras $\MA\subseteq \C^{n\times n}$ and $\MB\subseteq \C^{m\times m}$, we can construct the \emph{direct sum} $\MA\oplus \MB$ and \emph{tensor product} $\MA\otimes \MB$ defined by
\begin{equation*}
\MA\oplus \MB=\left\{\left(\begin{smallmatrix}A&0\\0&B\end{smallmatrix}\right) :  A\in \MA, B\in \MB\right\},
\end{equation*}
\begin{equation*}
\MA\otimes \MB=\left\{\sum_{i=1}^kA_i\otimes B_i : k\in \mathbb{N}, A_i\in \MA, B_i\in \MB\right\},
\end{equation*}
where $A\otimes B\in \C^{nm\times nm}$ denotes the Kronecker- or tensor product. The \emph{commutant} of $\MA$ is the matrix $*$-algebra
\begin{equation*}
\MA' =\{B\in \C^{n\times n} :  BA=AB \text{ for all $A\in\MA$}\}.
\end{equation*} 

Many interesting examples of matrix $*$-algebras come from unitary group representations, as we already demonstrated in the introduction: Given a unitary representation $\pi:G\to \Gl(n,\C)$, the set of invariant matrices $(\C^{n\times n})^G =\{A\in \C^{n\times n} : \pi(g)A\pi(g)^{-1}=A\}$ is a matrix $*$-algebra. It is the commutant of the matrix $*$-algebra linearly spanned by the matrices $\pi(g)$ with $g \in G$. If the unitary group representation is given by permutation matrices then the canonical basis of the algebra $(\C^{n \times n})^G$ are the zero-one incidence matrices of orbits on pairs $C_1, \ldots, C_M$, see Step 1 in Section~\ref{ssec:sdp invariant}.

Other examples of matrix $*$-algebras, potentially not coming from groups, include the (complex) Bose-Mesner algebra of an \emph{association scheme}, see e.g.\ Bannai, Ito \cite{BannaiIto}, and Brouwer, Cohen, Neumaier \cite{BrouwerCohenNeumaier} and more generally, the adjacency algebra of a \emph{coherent configuration}, see e.g.\ Cameron \cite{Cameron}. 

\subsection{Commutative matrix $*$-algebras}

A matrix $*$-algebra $\MA$ is called \emph{commutative} (or \emph{Abelian}) if any pair of its elements commute: $AB=BA$ for all $A,B\in \MA$. Recall that a matrix $A$ is \emph{normal} if $AA^*=A^*A$. The spectral theorem for normal matrices states that if $A$ is normal, there exists a unitary matrix $U$ such that $U^*AU$ is a diagonal matrix. More generally, a set of commuting normal matrices can be simultaneously diagonalized (see e.g.\ Horn, Johnson \cite{HornJohnson}). Since any algebra of diagonal matrices has a basis of zero-one diagonal matrices with disjoint support, we have the following theorem. 

\begin{mytheorem}\label{spectral}
Let $\MA\subseteq\C^{n\times n}$ be a commutative matrix $*$-algebra. Then there exist a unitary matrix $U$ and a partition $[n]=S_0\cup S_1\cup\cdots\cup S_k$ with $S_1,\ldots,S_k$ nonempty, such that 
\begin{equation*}
U^*\MA U=\{\lambda_1I_1+\cdots+\lambda_kI_k : \lambda_1,\ldots,\lambda_k\in \C\},
\end{equation*}
where $I_i$ is the zero-one diagonal matrix with ones in positions from $S_i$ and zeroes elsewhere. 
\end{mytheorem}

The matrices $E_i =UI_iU^*$ satisfy $E_0+E_1+\cdots+E_k=I$, $E_iE_j=\delta_{ij}E_i$ and $E_i=E_i^*$. The matrices $E_1,\ldots, E_k$ are the \emph{minimal idempotents} of $\MA$ and form an orthogonal basis of $\MA$. Unless $S_0=\emptyset$, $E_0$ does not belong to $\MA$.

Geometrically, we have an orthogonal decomposition
\begin{equation*}
\C^n=V_0\oplus V_1\oplus\cdots\oplus V_k,
\end{equation*} 
where $E_i$ is the orthogonal projection onto $V_i$ or equivalently, $V_i$ is the space spanned by the columns of $U$ corresponding to $S_i$. The space $V_0$ is the maximal subspace contained in the kernel of all matrices in $\MA$.

\subsection{Positive semidefinite elements}

Recall that a matrix $A$ is positive semidefinite ($A\succeq 0$) if and only if $A$ is \emph{Hermitian} (that is $A^*=A$) and all eigenvalues of $A$ are nonnegative. Equivalently, $A=U^*DU$ for some unitary matrix $U$ and nonnegative diagonal matrix $D$.

Let $\MA$ be a matrix $*$-algebra. Positive semidefiniteness can be characterized in terms of $\MA$.   
\begin{myproposition}\label{BstarB}
An element $A\in \MA$ is positive semidefinite if and only if $A=B^*B$ for some $B\in \MA$. 
\end{myproposition}

\begin{proof}
The `if' part is trivial. To see the `only if' part, let $A\in \MA$ be positive semidefinite and write $A=U^*DU$ for some unitary matrix $U$ and (nonnegative real) diagonal matrix $D$. Let $p$ be a polynomial with $p(\lambda_i)=\sqrt{\lambda_i}$ for all eigenvalues $\lambda_i$ of $A$. Then taking $B=p(A)\in \MA$ we have $B=U^*p(D)U$ and hence $B^*B=U^*p(D)p(D)U=U^*DU=A$.
\smartqed\qed
\end{proof}

Considered as a cone in the space of Hermitian matrices in $\MA$, the cone of positive semidefinite matrices is self-dual:

\begin{mytheorem}
Let $A\in \MA$ be Hermitian.  Then $A\succeq 0$ if and only if $\prodtrace{A}{B}\geq 0$ for all $B\succeq 0$ in $\MA$.
\end{mytheorem}

\begin{proof}
Necessity is clear. For sufficiency, let $A\in \MA$ be Hermitian and let $\MB\subseteq \MA$ be the $*$-subalgebra generated by $A$. By Theorem \ref{spectral} we can write $A=\lambda_1E_1+\cdots+\lambda_kE_k$, where the $E_i$ are the minimal idempotents of $\MB$. The $E_i$ are positive semidefinite since their eigenvalues are zero or one. Hence by assumption, $\lambda_i =\tfrac{\prodtrace{A}{E_i}}{\prodtrace{E_i}{E_i}}\geq 0$. Therefore $A$ is a nonnegative linear combination of positive semidefinite matrices and hence positive semidefinite. 
\smartqed\qed
\end{proof}

As a corollary we have:

\begin{mycorollary}
The orthogonal projection $\pi_{\MA}:\C^{n\times n}\to\MA$ preserves positive semidefiniteness. 
\end{mycorollary}

This implies that if the input matrices of the semidefinite program \eqref{sdp standard} lie in some matrix $*$-algebra, then we can assume that the optimization variable $X$ lies in the same matrix $*$-algebra: If \eqref{sdp standard} is given by matrices $C,A_1,\ldots,A_m\in \MA$ for some matrix $*$-algebra $\MA$, the variable $X$ may be restricted to $\MA$ without changing the objective value. Indeed, any feasible $X$ can be replaced by $\pi_{\MA}(X)$, which is again feasible and has the same objective value. When $\MA$ is the invariant algebra of a group, this amounts to replacing $X$ by the average under the action of the group.
 
\subsection{$*$-Homomorphisms and block diagonalization}
\label{sec:hom star}

\begin{mydefinition}
A map $\phi:\MA\to \MB$ between two matrix $*$-algebras $\MA$ and $\MB$ is called a \emph{$*$-homomorphism} if 
\begin{enumerate}
\item[(i)] $\phi$ is linear,
\item[(ii)] $\phi(AB)=\phi(A)\phi(B)$ for all $A,B\in \MA$,
\item[(iii)] $\phi(A^*)=\phi(A)^*$ for all $A\in \MA$.
\end{enumerate}
If $\phi$ is a bijection, the inverse map is also a $*$-homomorphism and $\phi$ is called a \emph{$*$-isomorphism}.
\end{mydefinition}

It follows directly from Proposition \ref{BstarB} that $*$-homomorphisms preserve positive semidefiniteness:
Let $\phi:\MA\to \MB$ be a $*$-homomorphism. Then $\phi(A)$ is positive semidefinite for every positive semidefinite $A\in \MA$.

The implication of this for semidefinite programming is the following. Given a semidefinite program with matrix variable restricted to a matrix $*$-algebra $\MA$ and a $*$-isomorphism $\MA\to \MB$, we can rewrite the semidefinite program in terms of matrices in $\MB$. This can be very useful if the matrices in $\MB$ have small size compared to those in $\MA$. In the following, we will discuss two such efficient representations of a general matrix $*$-algebra.

\subsection{Block diagonalization}
\label{sub:Blockdiag}

In this section, we study the structure of matrix $*$-algebras in some more detail. The main result is, that any matrix $*$-algebra is $*$-isomorphic to a direct sum of full matrix algebras:

\begin{mytheorem}\label{thm:blockdiagonal}
\textbf{(= Theorem~\ref{thm:block diagonal star})}
Let $\mathcal{A} \subseteq \C^{n \times n}$ be a matrix $*$-algebra. There are numbers $d$, and $m_1, \ldots, m_d$ so that there is $*$-isomorphism between $\mathcal{A}$ and a direct sum of complete matrix algebras
\begin{equation*}
\varphi : \mathcal{A} \to \bigoplus_{k=1}^d \C^{m_k \times m_k}.
\end{equation*}
\end{mytheorem}

This theorem is well-known in the theory of $C^*$-algebras, where it is generalized to $C^*$-algebras of compact operators on a Hilbert space (cf. Davidson \cite[Chapter I.10]{Davidson}). 

Some terminology: Let $\MA$ be a matrix $*$-algebra. A matrix $*$-algebra $\MB$ contained in $\MA$ is called a \emph{$*$-subalgebra} of $\MA$. An important example is the \emph{center} of $\MA$ defined by
\begin{equation*}
\mathcal{C}_{\MA} =\{A\in \MA :  AB=BA \text{ for all $B\in\MA$}\}.
\end{equation*}

There is a unique element $E\in \MA$ such that $EA=AE=A$ for every $A\in \MA$, which is called the \emph{unit} of $\MA$. 
If $\MA$ is non-zero and $\mathcal{C}_{\MA}=\C E$ (equivalently: $\MA$ has no nontrivial \emph{ideal}), the matrix $*$-algebra $\MA$ is called \emph{simple}.  

We shall prove that every matrix $*$-algebra is the direct sum of simple matrix $*$-algebras:
\begin{equation*}
\MA=\bigoplus_{i=1}^d E_i\MA,
\end{equation*}
where the $E_i$ are the minimal idempotents of $\mathcal{C}_{\MA}$. Every simple matrix $*$-algebra is $*$-isomorphic to a full matrix algebra. Together these facts imply Theorem \ref{thm:blockdiagonal}.

To conclude this section, we will give an elementary and detailed proof of Theorem \ref{thm:blockdiagonal}.
\begin{proof} 
Let $\MB\subseteq \MA$ be an inclusionwise maximal commutative $*$-subalgebra of $\MA$. Then any $A\in \MA$ that commutes with every element of $\MB$, is itself an element of $\MB$. Indeed, if $A$ is normal, this follows from the maximality of $\MB$ since $\MB\cup \{A,A^*\}$ generates a commutative $*$-algebra containing $\MB$. If $A$ is not normal, then $A+A^*\in \MB$ by the previous argument and hence $A(A+A^*)=(A+A^*)A$, contradicting the fact that $A$ is not normal.

By replacing $\MA$ by $U^*\MA U$ for a suitable unitary matrix $U$, we may assume that $\MB$ is in diagonal form as in Theorem \ref{spectral}. For any $A\in \MA$ and $i,j=0,\ldots,k$, denote by $A_{ij}\in \C^{|S_i|\times |S_j|}$ the restriction of $I_iAI_j$ to the rows in $S_i$ and columns in $S_j$. Let $A\in \MA$ and $i,j\in \{0,\ldots,k\}$. We make the following observations: 
\begin{enumerate}
\item[(i)] $A_{ii}$ is a multiple of the identity matrix and $A_{00}=0$,
\item[(ii)] $A_{0i}$ and $A_{i0}$ are zero matrices,
\item[(iii)] $A_{ij}$ is either zero or a nonzero multiple of a (square) unitary matrix.
\end{enumerate}
Item (i) follows since $I_iAI_i$ commutes with $I_0,\ldots,I_k$ and therefore belongs to $\MB$. Hence $I_iAI_i$ is a multiple of $I_i$ and $I_0AI_0=0$ since $I_0\MB I_0=\{0\}$. Similarly, $I_0AA^*I_0=0$, which implies that $I_0A=0$, showing (ii). For item (iii), suppose that $A_{ij}$ is nonzero and assume without loss of generality that $|S_i|\geq |S_j|$. Then by (i), $A_{ij}A_{ij}^*=\lambda I$ for some positive real $\lambda$, and therefore has rank $|S_i|$. This implies that $|S_j|=|S_i|$ and $\sqrt{\lambda}\cdot A_{ij}$ is unitary.

Observe that (ii) shows that $I_1+\cdots+I_k$ is the unit of $\MA$.

Define the relation $\sim$ on $\{1,\ldots,k\}$ by setting $i\sim j$ if and only if $I_i\MA I_j\neq \{0\}$. This is an equivalence relation. Indeed, $\sim$ is reflexive by (i) and symmetric since $I_i\MA I_j=(I_j\MA I_i)^*$. Transitivity follows from (iii) since $I_h\MA I_j\supseteq (I_h\MA I_i)(I_i\MA I_j)$ and the product of two unitary matrices is unitary. 

Denote by $\{E_1,\ldots,E_d\}=\{\sum_{j\sim i} I_j :  i=1,\ldots, k\}$ the zero-one diagonal matrices induced by the equivalence relation. Since the center $\mathcal{C}_{\MA}$ of $\MA$ is contained in $\MB=\C I_1+\cdots+\C I_k$, it follows by construction that $E_1,\ldots, E_d$ span the center, and are its minimal idempotents. We find the following block structure of $\MA$:
\begin{equation*}
\MA=\{0\}\oplus E_1\MA\oplus\cdots\oplus E_d\MA,
\end{equation*}
where the matrix $*$-algebras $E_i\MA$ are simple. For the rest of the proof we may assume that $\MA$ is simple ($d=1$) and that $E_0=0$. 

Since $\sim$ has only one equivalence class, for every matrix $A=(A_{ij})_{i,j=1}^k\in \MA$, the `blocks' $A_{ij}$ are square matrices of the same size. Furthermore, we can fix an $A\in \MA$ for which all the $A_{ij}$ are unitary.  For any $B\in \MA$, we have $A_{1i}B_{ij}(A_{1j})^*=(AI_iBI_jA^*)_{11}$. By (i), it follows that $\{A_{1i}B_{ij}(A_{1j})^* : B\in \MB\}=\C I$. Hence setting $U$ to be the unitary matrix $U:=\diag (A_{11},\ldots, A_{1k})$, we see that $U\MA U^*=\{\left(a_{ij}I\right)_{i,j=1}^k :  a_{i,j}\in \C\}$, which shows that $\MA$ is $*$-isomorphic to $\C^{k\times k}$.
\smartqed\qed
\end{proof}

\subsection*{Relation to group representations}

In the case that $\MA=(\C^{n\times n})^G$, where $\pi:G\to U_n(\C)$ is a unitary representation, the block diagonalization can be interpreted as follows. The diagonalization of the maximal commutative $*$-subalgebra $\MB$, gives a decomposition $\C^n=V_1\oplus\cdots\oplus V_k$ into irreducible submodules. Observe that $V_0=\{0\}$ since $\MA$ contains the identity matrix. The equivalence relation $\sim$ yields the isotypic components $\mathrm{Im} E_1,\ldots,\mathrm{Im} E_d$, where the sizes of the equivalence classes correspond to the block sizes $m_i$ in Theorem \ref{thm:blockdiagonal}. Here Schur's lemma is reflected by the fact that $I_i\MA I_j=\C I_i$ if $i\sim j$ and $\{0\}$ otherwise. 

\subsection*{Algorithmic aspects}

If a matrix $*$-algebra $\MA$ is given explicitly by a basis, then the above proof of Theorem \ref{thm:blockdiagonal} can be used to find a block diagonalization of $\MA$ computationally. Indeed, it suffices to find an inclusion-wise maximal commutative $*$-subalgebra $\MB\subseteq \MA$ and compute a common system of eigenvectors for (basis) elements of $\MB$. This can be done by standard linear algebra methods. For example finding a maximal commutative $*$-subalgebra of $\MA$ can be done by starting with $\MB=\left<A\right>$, the $*$-subalgebra generated by an arbitrary Hermitian element in $\MA$. As long as $\MB$ is not maximal, there is a Hermitian element in $\MA\setminus \MB$ that commutes with all elements in (a basis of) $\MB$, hence extending $\MB$. Such an element can be found by solving a linear system in $O(\dim \MA)$ variables and $O(\dim \MB)$ constraints. In at most $\dim\MA$ iterations, a maximal commutative $*$-subalgebra is found.

Practically more efficient is to find a ``generic'' Hermitian element $A\in \MA$. Then the matrix $*$-algebra $\MB$ generated by $A$ will be a maximal commutative $*$-subalgebra of $\MA$ and diagonalizing the matrix $A$ also diagonalizes $\MB$. Such a generic element can be found by taking a random Hermitian element from $\MA$ (with respect to the basis), see Murota et. al. \cite{MurotaKannoKojimaKojima}. If a basis for the center of $\MA$ is known a priori (or by solving a linear system in $O(\dim\MA)$ variables and equations), as an intermediate step the center could be diagonalized, followed by a block diagonalization of the simple components of $\MA$, see Dobre, de Klerk, Pasechnik \cite{DobredeKlerkPasechnik}.       

\subsection{Regular $*$-representation}\label{sub:Regularstar}

Let $\MA$ be a matrix $*$-algebra of dimension $M$ and let $C_1,\ldots,C_M$ be an orthonormal basis of $\MA$ with respect to the trace product $\prodtrace{\cdot}{\cdot}$. For fixed $A\in \MA$, left-multiplication by $A$ defines a linear map  $B\mapsto AB$ on $\MA$. With respect to the orthonormal basis $C_1,\ldots,C_M$, this linear map is represented by the matrix $L(A)\in \C^{M\times M}$ given by
\begin{equation*}
L(A)_{st} =\prodtrace{AC_t}{C_s}.
\end{equation*}
The map
\begin{equation*}
L:\MA\to \C^{M\times M} 
\end{equation*}
is an injective $*$-homomorphism called the \emph{regular $*$-representation of $\MA$}. The fact that $L$ is linear and preserves matrix products is clear. Injectivity follows from the fact that $L(A)=0$ implies $AA^*=0$ and hence $A=0$. Finally, the equations  
\begin{equation*}
L(A^*)_{st}=\prodtrace{A^*C_t}{C_s}=\prodtrace{C_t}{AC_s}=\overline{\prodtrace{AC_s}{C_t}}=\overline{L(A)_{ts}}
\end{equation*} 
show that $L(A^*)=L(A)^*$. Because $L$ is linear, it is determined by the images $L(C_1),\ldots,L(C_M)$. 

In many applications, for example in the case of a permutation action with the canonical basis $C_1, \ldots, C_M$ (Step $1\frac{1}{2}$ in the introduction), one only has an orthogonal basis which is not orthonormal. In that case, the map $L$ is given by
\begin{equation*}
L(C_r)_{st}=\frac{\prodtrace{C_rC_t}{C_s}}{||C_s|| \; ||C_t||}=\frac{||C_s||}{||C_t||} p_{rt}^s,
\end{equation*}   
where the $p_{rs}^t$ are the multiplication parameters defined by $C_rC_s=\sum_t p_{rs}^t C_t$.  If we denote $\phi(C_r) = L(C_r)$, we obtain Theorem \ref{thm:regular representation}.

\section{Invariant positive definite functions on compact spaces}
\label{Representation theory}

Until now we considered only finite dimensional invariant semidefinite
programs and the question: 
How can symmetry coming from the action of a finite group be exploited
to simplify the semidefinite program? 
In several situations, some given in Section~\ref{Spherical codes},
one wants to work with infinite dimensional invariant semidefinite
programs. 
In these situations using the symmetry coming from the action of an
infinite, continuous group is a must if one wants to do explicit
computations. 
In this section we introduce the cone of continuous, positive definite
functions on a compact set $M$, as a natural generalization of the
cone of positive semidefinite matrices. When a compact group $G$ acts
continuously on $M$, we use the representation theory of $G$ to
describe the subcone of $G$-invariant positive definite functions on
$M$. This is the infinite dimensional analog of Step 2 (second
version), given in the introduction, in the case of a permutation
representation.

In general, this method evidences interesting links between the
geometry of $M$, the representations of $G$ and the theory of
orthogonal polynomials. We will review the spaces $M$, finite and infinite, where this
method has been  completely worked out.

\subsection{Positive definite functions}
\label{Pdf}

Let $M$ be a compact space. We denote the
space of continuous functions on $M$ taking complex values by
$\CC(M)$. 

\begin{mydefinition}
We say that $F\in \CC(M^2)$ is \emph{positive definite}, $F\succeq 0$, if
$F(x,y)=\overline{F(y,x)}$ and, for all $n$, for all $n$-tuples
$(x_1,\dots,x_n)\in M^n$ and vectors $(\alpha_1,\dots,\alpha_n)\in \C^n$,
\begin{equation*}
\sum_{i,j=1}^n \alpha_i F(x_i,x_j)\overline{\alpha_j}\geq 0.
\end{equation*}
In other words, for all choices of finite subsets $\{x_1,\dots,x_n\}$ of
$M$, the matrix 
\begin{equation*}
\big(F(x_i,x_j)\big)_{1\leq i,j\leq n}
\end{equation*}
is Hermitian and positive semidefinite. The set  of continuous positive definite
functions on $M$ is denoted $\CC(M^2)_{\succeq 0}$.
\end{mydefinition}
In particular, if $M$ is a finite set with the discrete topology, the
elements of $\CC(M^2)_{\succeq 0}$ coincide with  the Hermitian
positive semidefinite matrices indexed by $M$.
In general, $\CC(M^2)_{\succeq 0}$ is  a closed convex cone in $\CC(M^2)$.

We now assume that a compact group $G$ \emph{acts continuously on
  $M$}. Here we mean that, in addition to the properties of a group
action (see introduction), the map $G\times M \to M$, with $(g,x) \mapsto gx$, is continuous. 
We also assume that $M$ is endowed with a \emph{regular Borel measure} $\mu$ which is
\emph{$G$-invariant}.

\begin{verse}
\small
\noindent We recall here the  basic notions on measures that will be needed.
If $M$ is a topological space, a \emph{Borel measure} is a measure defined on the $\sigma$-algebra
generated by the open sets of $M$. A  \emph{regular Borel measure} is
one which behaves well with respect to this topology, we refer to Rudin
\cite{Rudin} for the precise definition.

\noindent If $M$ is a finite set, the topology on $M$ is chosen to be the \emph{discrete
topology}, i.e. every subset of $M$ is an open set, and the measure $\mu$ on $M$ is the
\emph{counting measure} defined by $\mu(A)=|A|$.

\noindent On a compact group $G$, there is a regular Borel measure $\lambda$ which is {\em
  left and right invariant}, meaning that
$\lambda(gA)=\lambda(Ag)=\lambda(A)$ for all $g\in G$ and $A\subset
G$ measurable, and such that the open sets have positive measure. This
measure is unique up to a positive multiplicative constant and is
called the \emph{Haar measure on $G$} (see e.g. Conway \cite[Theorem 11.4]{Conway}). If $G$ is finite, the counting
measure on $G$ provides a  measure with these properties. 
\end{verse}

\subsection{Unitary representations of compact groups}\label{Unitary rep}

The definitions and properties of representations given in the
introduction  extend naturally to compact groups. We only have to
modify the definition of a finite dimensional representation by asking
that the homomorphism $\pi: G \to \Gl(V)$ is continuous. 

In the proofs,
integrating with the \emph{Haar measure} $\lambda$ of $G$ replaces averaging on finite
groups. For example, in the introduction, we have seen 
that every representation $V$ of a finite group $G$ is unitary. The argument was, starting from
an arbitrary
inner product $\langle u,v\rangle_0$ on $V$, to average it over all
elements of $G$, in order to transform it into a $G$-invariant inner
product. In the case of a more general compact group, we average
over $G$ using the Haar measure (normalized so that $\lambda(G)=1$), thus taking:
\begin{equation*}
\langle u,v\rangle :=\int_{G}\langle gu,gv\rangle d\lambda(g).
\end{equation*}
So Maschke's theorem holds: every finite dimensional $G$-space is the direct sum of
irreducible $G$-subspaces. In contrast to the case of finite groups,
the number of isomorphism classes of irreducible $G$-spaces is generally
infinite. We fix a set $\RR=\{R_k : k\geq 0\}$ of
representatives, assuming for simplicity that this set is countable
(it is the case e.g. if $G$ is a Lie group). Here $R_0$ is the trivial
representation, i.e. $V=\C$ with $gz=z$ for all $g\in G$, $z\in
\C$. 

However, infinite dimensional representations arise naturally in the study
of infinite compact groups. We shall be primarily concerned with one of them,
namely the representation of $G$ defined by the space $\CC(M)$ 
with the action of $G$ defined by $\pi(g)(f)(x) =f(g^{-1}x)$ (of
course it has infinite dimension only if $M$ itself is not a finite set). It is a
unitary
representation for the standard inner product on $\CC(M)$:
\begin{equation*}
\langle f_1,f_2\rangle =\frac{1}{\mu(M)}\int_M
f_1(x)\overline{f_2(x)}d\mu(x).
\end{equation*}
Moreover $\CC(M)$ is a dense subspace of the \emph{Hilbert space} 
$L^2(M)$ of measurable functions $f$ on $M$ such that $|f|^2$ is $\mu$-integrable (see
e.g. Conway \cite[Chapter 1]{Conway} or Rudin \cite[Chapter 4]{Rudin}).
It turns out that the theory of finite dimensional representations of
compact groups extends nicely to the unitary representations on Hilbert
spaces:
\begin{mytheorem}\label{Th2}
Every unitary representation of a compact group $G$ on a Hilbert space
is the direct  sum (in the sense of Hilbert spaces) of finite dimensional irreducible $G$-subspaces.
\end{mytheorem}
The above theorem is in fact a consequence of the celebrated \emph{Theorem
of Peter and Weyl} (see e.g. Bump \cite{Bump}), which states that the matrix coefficients of finite
dimensional irreducible representations of a compact group $G$ span a
subspace of $\CC(G)$ which is dense for the topology of uniform
convergence.

\subsection{$G$-invariant positive definite functions on
  $M$}\label{G-invariant pdf}

Now we come to our main concern, which is to describe the elements of
$\CC(M^2)^G_{\succeq 0}$, i.e. the $G$-invariant positive definite
functions, in terms of the structure of the $G$-space $M$. In order to avoid technicalities with
  convergence issues, we shall systematically work in finite
  dimensional subspaces of $\CC(M)$. So let $V$ be a $G$-subspace of
  $\CC(M)$ of finite dimension. Let $V^{(2)}$ denote the subspace  of $\CC(M^2)$ spanned by elements of the form
$f_1(x)\overline{f_2(y)}$, where $f_1$ and $f_2$ belong to $V$.

We take the following notations for an irreducible decomposition of
$V$: 
\begin{equation}\label{decomp}
V=\bigoplus_{k\in I_V} \bigoplus_{i=1}^{m_k} H_{k,i},
\end{equation}
where for $1\leq i\leq m_k$, the subspaces $H_{k,i}$ are pairwise
orthogonal and $G$-iso\-mor\-phic to $R_k$, and where $I_V$ is the finite set of indices $k\geq 0$ such that
the multiplicities $m_k$ are not equal to $0$. 
For all $k$, $i$, we choose an
orthonormal basis $(e_{k,i,1},\dots, e_{k,i,d_k})$ of $H_{k,i}$, where
$d_k=\dim(R_k)$, such that  the
complex numbers $\langle ge_{k,i,s},e_{k,i,t}\rangle$ do not depend on
  $i$ (such a basis exists precisely because the $G$-isomorphism class of
  $H_{k,i}$ does not depend on $i$).
From this data, we introduce the  $m_k\times m_k$ matrices $E_k(x,y)$
with coefficients $E_{k,ij}(x,y)$:
\begin{equation}\label{Ek}
E_{k,ij}(x,y) =\sum_{s=1}^{d_k} e_{k,i,s}(x)\overline{e_{k,j,s}(y)}.
\end{equation}
Then $V^{(2)}$ obviously contains the elements
$e_{k,i,s}(x)\overline{e_{k,j,s}(y)}$,  which moreover form an orthonormal
basis  of this space.

\begin{mytheorem}\label{Th3}
Let $F\in V^{(2)}$. Then  $F$ is a $G$-invariant positive definite
function
if and only if there exist Hermitian positive semidefinite matrices
$F_k$ such that
\begin{equation}\label{pdf}
F(x,y)=\sum_{k\in I_V} \langle F_k, \overline{E_k(x,y)}\rangle.
\end{equation}
\end{mytheorem}

Before we give a sketch of the proof, several comments are in order:
\begin{enumerate}
\item If $M = [n]$, then we recover the parametrization of
  $G$-invariant positive semidefinite matrices given in
  \eqref{eq:parameterization}.
\item The main point of the above theorem is
  to replace the property of a continuous function being positive
  definite with the property that the $m_k\times m_k$ matrices $F_k$ are
  positive semidefinite. In the introduction, we have already seen 
in the finite case $M=[n]$ that it allows  to reduce  the size of 
invariant semidefinite
  programs (following Step 2 (second version) of the introduction). In
  the general case, this theorem allows to
replace  a conic linear program
  involving the cone $\CC(M^2)_{\succeq 0}$ with  standard
  semidefinite programs (see Section \ref{Spherical codes} for examples where $M$ is
  the unit sphere of Euclidean space).
\item One might wonder to what extent the above matrices
  $E_k(x,y)$ depend on the choices involved in their
  definition. Indeed, one can prove that a different choice of
  orthonormal basis in $H_{k,i}$ does not affect $E_k(x,y)$. One could
  also start from another decomposition of the isotypic component
  $\MI_k$ of $V$. The new decomposition does not need to be orthogonal; the matrix $E_k(x,y)$ would then  change to
  $AE_k(x,y)A^*$ for some $A\in \Gl_{m_k}(\C)$. So, up to an action
  of $\Gl_{m_k}(\C)$, the matrix
$E_k(x,y)$ is canonically associated to $\MI_k$.
\item The statement implies that the matrices $E_k(x,y)$ themselves are
  $G$-invariant. In other words, they can be expressed in terms of the
  orbits
$O(x,y)$ of the action of $G$ on $M^2$ in the form:
\begin{equation}\label{Yk}
E_k(x,y)=Y_k(O(x,y))
\end{equation}
for some matrices $Y_k$. It is indeed the expression we are
  seeking for.

\end{enumerate}

\begin{proof} $F(x,y)$ is a linear
  combination of the  $e_{k,i,s}(x)\overline{e_{l,j,t}(y)}$ since these elements
  form an orthonormal basis of $V^{(2)}$. In
  a first step, one shows that the $G$-invariance property $F(gx,gy)=F(x,y)$ results
  in an expression for $F$ of the form \eqref{pdf} for some matrices
  $F_k=(f_{k,ij})$; the proof involves the \emph{orthogonality
  relations}  between matrix coefficients of   
irreducible representations of $G$ (see e.g. Bump \cite{Bump}).
In a second step, $F\succeq 0$ is shown to be equivalent to
$F_k\succeq 0$ for all $k\in I_V$. Indeed, 
for $\alpha(x) =\sum_{i=1}^{m_k} \alpha_i \overline{e_{k,i,s}(x)}$, 
we have 
\begin{equation*}
\sum_{i,j=1}^{m_k} \alpha_i f_{k,ij} \overline{\alpha_j}=\frac{1}{\mu(M)^2}\int_{M^2}
\alpha(x)F(x,y)\overline{\alpha(y)}d\mu(x,y)\geq 0.
\end{equation*}
\smartqed\qed
\end{proof}

\begin{myremark}
\begin{enumerate}
\item A straightforward generalization of the main structure theorem for matrix $*$-algebras to the $*$-algebra $\End^G(V)$ shows  that 
\begin{equation}\label{isom}
\End^G(V)\simeq (V^{(2)})^G\simeq \bigoplus_{k\in I_V} \C^{m_k\times m_k}.
\end{equation}
An isomorphism from $(V^{(2)})^G$ to the direct
sum of matrix algebras is constructed in the proof of Theorem \ref{Th3}, in an explicit way, from the
decomposition \eqref{decomp}. This isomorphism  is given by the map
sending $(F_k)_{k\in I_V}\in \oplus_{k\in I_V} \C^{m_k\times m_k}$
  to $F\in (V^{(2)})^G$ defined by:
\begin{equation*}
F(x,y)=\sum_{k\in I_V} \langle F_k,\overline{ E_k(x,y)}\rangle
\end{equation*}
thus completing Step 2 (second version) of the introduction
for compact groups.

\item Theorem \ref{Th3} shows, jointly with Theorem
  \ref{Th2}, that any element of the cone $\CC(M^2)_{\succeq 0}^G$ is a sum
  (in the sense of $L^2$ convergence) of the form \eqref{pdf} with
  possibly infinitely many terms. Moreover, under the additional
  assumption that $G$ acts transitively on $M$, Bochner \cite{Bochner} proves that
  the convergence holds in the stronger sense of \emph{uniform
    convergence}, i.e. for the supremum  norm.
\end{enumerate}
\end{myremark}

\subsection{Examples: the commutative case}\label{Ex: commutative}
We mean here the case when $\End^G(\CC(M))$ is a commutative algebra. From
Theorem \ref{Th3}, this condition is equivalent to the commutativity
of $\End^G(V)$ for all finite dimensional $G$-subspace of
$\CC(M)$. So, from the
isomorphism \eqref{isom}, and since a matrix algebra $\C^{m\times m}$ is
commutative if and only if $m=0$ or $m=1$,  the commutative case
corresponds  to non vanishing 
multiplicities equal to $1$ in the decomposition of $\CC(M)$. The
$G$-space $\CC(M)$ is said to be multiplicity-free. Then the matrices $F_k$ in \eqref{pdf}
have only one coefficient $f_k$ and $F\in \CC(M^2)_{\succeq 0}^G$ is of
the form
\begin{equation}\label{pdf sym}
F(x,y)=\sum_{k\geq 0} f_k E_k(x,y) ,\quad f_k\geq 0.
\end{equation}
We say that $M$ is \emph{$G$-symmetric}, if $G$ acts \emph{transitively on
$M$} (i.e. there is only one orbit in $M$; equivalently $M$ is said to
be a \emph{homogeneous space}), and
if moreover, for all $(x,y)\in M^2$, there exists $g\in G$ such that
$gx=y$ and $gy=x$. In other words $(y,x)$ belongs to the $G$-orbit
$O(x,y)$ of the pair $(x,y)$.  This nonstandard terminology (sometimes it is also called a \emph{generously transitive} group action) covers the case of
\emph{compact symmetric Riemannian manifolds} (see e.g. Berger \cite{Berger}) as well as a large number of
finite spaces. We provide many examples of such spaces below. The
functions $Y_k$ \eqref{Yk}
are often referred to as the \emph{spherical functions} or
\emph{zonal spherical functions} of the homogeneous space $M$ (see e.g. Vilenkin, Klimyk \cite{VilenkinKlimyk}).

\begin{mylemma} If $M$ is $G$-symmetric then for all $V\subset \CC(M)$,
  $(V^{(2)})^G$ is commutative.
\end{mylemma}
\begin{proof} Since $F\in (V^{(2)})^G$ is $G$-invariant and $M$ is $G$-symmetric, 
$F(x,y)=F(y,x)$. The multiplication on the algebra $(V^{(2)})^G\simeq \End^G(M)$
  is the convolution $F*F'$ of functions, given by:
\begin{equation*}
(F*F')(x,y)=\int_M F(x,z)F'(z,y)d\mu(z).
\end{equation*}
A straightforward computation shows that $F*F'=F'*F$.
\smartqed\qed
\end{proof}
An important subclass of symmetric spaces is provided by the 
\emph{two-point homogeneous spaces}. These spaces are metric spaces, with
distance
$d(x,y)$ invariant by $G$, and moreover satisfy that two pairs of
elements $(x,y)$ and $(x',y')$ belong to the same $G$-orbit if and
only if $d(x,y)=d(x',y')$. They are obviously $G$-symmetric since
$d(x,y)=d(y,x)$.
In other words, the distance function parametrizes the orbits of $G$
acting on pairs, thus the expression \eqref{pdf sym} of positive
definite functions specializes to
\begin{equation}\label{pdf 2-pt hom}
F(x,y)=\sum_{k\geq 0} f_k Y_k(d(x,y)) ,\quad f_k\geq 0,
\end{equation}
for some functions $Y_k$ in one variable. The functions $Y_k$ are explicitly
known for many spaces, and are usually given by a certain family of
\emph{orthogonal polynomials} (see e.g. Szeg\"o \cite{Szego}, Andrews,
Askey, Roy \cite{AndrewsAskeyRoy}). The orthogonality property arises
because the associated irreducible subspaces are pairwise orthogonal.

The complete list of
real, compact, connected two-point homogeneous spaces  was established by Wang \cite{Wang}: the unit sphere of
Euclidean space $S^{n-1}\subset \R^n$, which is treated in Section~\ref{Spherical codes}, the
projective spaces over the fields of real, complex and quaternion
numbers (denoted $P^{n-1}(K)$, with respectively $K=\R$, $\C$, $\HH$), and
the projective plane over the octonions $P^2(\OO)$, are the only spaces having these properties. In each of
these cases the functions $Y_k$ are given by \emph{Jacobi
  polynomials} (see Table~1). 
These polynomials depend on two parameters $(\alpha,\beta)$ and are
orthogonal for the measure $(1-t)^{\alpha}(1+t)^{\beta}dt$ on the
interval $[-1,1]$. We denote by $P_k^{(\alpha,\beta)}(t)$ the Jacobi polynomial of degree $k$,
normalized by the condition $P_k^{(\alpha,\beta)}(1)=1$. When
$\alpha=\beta=\lambda-1/2$, these polynomials are equal (up to a nonnegative
multiplicative factor) to the \emph{Gegenbauer polynomials
  $C_k^{\lambda}(t)$}. 

\begin{table}\label{Table 1}
\begin{center}
\begin{tabular}{llcc}
Space & Group & $\alpha$ & $\beta$\\
\hline
 $S^{n-1}$ & $O_n(\R)$ &$(n-3)/2$ & $(n-3)/2$ \\
 $P^{n-1}(\R)$ & $O_n(\R)$ & $(n-3)/2$ & $-1/2$ \\
 $P^{n-1}(\C)$ & $U_n(\C)$ & $n-2$ & $0$ \\
 $P^{n-1}(\HH)$ & $U_n(\HH)$ & $2n-3$ & $1$\\
 $P^{2}(\OO)$ & $F_4(\R)$ & $7$ & $3$ \\
\hline
\end{tabular}
\end{center}
\caption{The real compact two-point homogeneous spaces, their groups and their
 spherical functions $Y_k$.}
\end{table}

The finite two-point homogeneous spaces are not completely classified,
but several of them are relevant spaces for coding theory. The most
prominent one is the Hamming space  $\q^n$,  associated to the Krawtchouk
polynomials, which is discussed in  Section~\ref{Block
  codes}. It is the space of $n$-tuples, over
a finite set $\q$ of $q$ elements. An element of the Hamming space
$\q^n$ is  usually called a word, and $\q$ is called the
alphabet. The closely related 
\emph{Johnson spaces} are the subsets of binary
words of fixed length and  weight. The Hamming and Johnson spaces are
endowed with the \emph{Hamming distance}, counting the number of
coordinates
where to words disagree.  The Johnson spaces have $q$-analogues, the
 \emph{$q$-Johnson spaces}, the subsets of linear subspaces of
$\F_q^n$ of fixed dimension. The distance between two subspaces of the
same dimension is measured by the difference between this dimension
and the dimension of their intersection. Other spaces are related to finite
geometries and to spaces of matrices. In the later case, the distance
between two matrices will be the rank of the difference of the matrices. 

The spherical functions have been worked out, initially
in the framework of association schemes with the work of
Delsarte \cite{Delsarte1}, and later  in relation with harmonic
analysis of finite classical groups with the work of Delsarte, Dunkl,
Stanton and others. We refer to Conway, Sloane
\cite[Chapter 9]{ConwaySloane}
and to Stanton \cite{Stanton} for  surveys on
these topics. Table~2 displays  the most important
families of these spaces, together with their groups of isometries and
the family of orthogonal polynomials to which their spherical
functions belong; see Stanton \cite{Stanton} for
additional  information.

\begin{table}\label{Table 2}
\begin{center}
\begin{tabular}{llll}
Space  & Group & Polynomial & Reference \\
\hline
Hamming space $\q^n$ & $S_q\wr S_n$ & Krawtchouk &\cite{Delsarte1}\\
Johnson space    & $S_n$ &  Hahn & \cite{Delsarte1},\cite{Delsarte2}\\
$q$-Johnson space  & $\Gl_n(\F_q)$&   $q$-Hahn & \cite{Delsarte2} \\
\multicolumn{4}{l}{\em Maximal totally isotropic subspaces of dimension $k$, for
  a nonsingular bilinear form: }\\
Symmetric     & $\operatorname{SO}_{2k+1}(\F_q)$ & $q$-Krawtchouk & \cite{Stanton}\\
  & $\operatorname{SO}_{2k}(\F_q)$ & $q$-Krawtchouk & \cite{Stanton}\\
 & $\operatorname{SO}_{2k+2}^-(\F_q)$ & $q$-Krawtchouk &\cite{Stanton}\\
 Symplectic    & $\operatorname{Sp}_{2k}(\F_q)$ & $q$-Krawtchouk & \cite{Stanton}\\
 Hermitian    & $\operatorname{SU}_{2k}(\F_{q^2})$ & $q$-Krawtchouk & \cite{Stanton}\\
  & $\operatorname{SU}_{2k+1}(\F_{q^2})$ & $q$-Krawtchouk & \cite{Stanton}\\
\multicolumn{4}{l}{\em Spaces of matrices: } \\
$\F_q^{k\times n} $ & $\F_q^{k\times  n}. (\Gl_k(\F_q)\times\Gl_n(\F_q))$ & Affine $q$-Krawtchouk &\cite{Delsarte3}, \cite{Stanton}\\
$\operatorname{Skew}_n(\F_q)$ skew-symmetric &$\operatorname{Skew}_n(\F_q).\Gl_n(\F_q)$ &Affine
$q$-Krawtchouk&\cite{DelsarteGoethals}, \cite{Stanton}\\
$\operatorname{Herm}_n(\F_{q^2})$ Hermitian   &$\operatorname{Herm}_n(\F_{q^2}).\Gl_n(\F_{q^2})$ &Affine $q$-Krawtchouk&\cite{Stanton}\\
\hline
\end{tabular}
\end{center}
\caption{Some finite two-point homogeneous spaces, their groups  and their
spherical  functions $Y_k$.}
\end{table}

Symmetric spaces which are not two-point homogeneous give rise to
functions $Y_k$ depending on several variables. A typical example is provided
by the real Grassmann spaces, the spaces of $m$-dimensional subspaces of
$\R^n$, $m\leq n/2$, under the action of the orthogonal group
$O(\R^n)$. Then, the functions $Y_k$ are multivariate Jacobi
polynomials, with $m$ variables representing the principal
angles between two subspaces (see James, Constantine \cite{JamesConstantine}). A similar situation occurs in coding
theory when nonhomogeneous alphabets are considered, involving
multivariate Krawtchouk polynomials. Table 3 shows some examples of
these spaces with references to coding theory applications. 

\begin{table}\label{Table 3}
\begin{center}
\begin{tabular}{llll}
Space & Group & Polynomial & Reference \\
\hline
Nonbinary Johnson &$S_{q-1}\wr S_n$ &Eberlein/Krawtchouk& \cite{TarnanenAaltonenGoethals}\\
Permutation group &$S_n^2$&Characters of $S_n$ & \cite{Tarnanen}\\
Lee space &$D_q\wr S_n$ &Multivariate Krawtchouk& \cite{Astola}\\
Ordered Hamming space &$(\F_q^k.B_k)\wr S_n$ &Multivariate
Krawtchouck&\cite{MartinStinson}, \cite{Barg}\\
Real Grassmann space & $O_n(\R)$ & Multivariate Jacobi &
\cite{JamesConstantine}, \cite{Bachoc1}\\
Complex Grassmann space & $U_n(\C)$ & Multivariate Jacobi &
\cite{JamesConstantine}, \cite{Roy1}, \cite{Roy2}\\
Unitary group  & $U_n(\C)^2$ & Schur & \cite{CreignouDiet}\\
\hline
\end{tabular}
\end{center}
\caption{Some symmetric spaces, their groups and their spherical
 functions $Y_k$.}
\end{table}

\subsection{Examples: the noncommutative case}\label{Ex: non commutative} 

Only a few cases have
been completely worked out. Among them, the Euclidean sphere $S^{n-1}$
under the action of the subgroup of the orthogonal group fixing one
point is treated in Bachoc, Vallentin \cite{BachocVallentin1}. We come back to this case in Section~\ref{Spherical codes}. These results have been extended to
the case of fixing many points in Musin \cite{Musin}. The case of the binary Hamming space
under the action of the symmetric group is treated in the framework of
group representations in Vallentin \cite{Vallentin2}, shedding a new light on the results
of Schrijver \cite{Schrijver2}. A very similar case is given by the set of linear
subspaces
of the finite vector space $\F_q^n$, treated as a $q$-analogue of the
binary Hamming space in Bachoc, Vallentin \cite{BachocVallentin3}.

\section{Block codes}
\label{Block codes}

In this section, we give an overview of symmetry reduction in semidefinite programming bounds for (error correcting) codes. 
We fix an alphabet $\q=\{0,1,\ldots,q-1\}$ for some integer $q\geq 2$. The \emph{Hamming space} $\q^n$ is equipped with a metric $d(\cdot,\cdot)$ called the \emph{Hamming distance} which is given by 
\begin{equation*}
d(u,v) =|\{i :  u_i\neq v_i\}|\text{ for all $u,v\in \q^n$.}
\end{equation*}
The isometry group of the Hamming space will be denoted $\Aut(q,n)$ and is a wreath product $S_q\wr S_n$ of symmetric groups. That is, the isometries are obtained by taking a permutation of the $n$ positions followed by independently permuting the symbols $0,\ldots,q-1$ at each of the $n$ positions. The group acts transitively on $\q^n$ and the orbits of pairs are given by the Hamming distance. 

A subset $C\subseteq \q^n$ is called a \emph{code} of length $n$. For a (nonempty) code $C$, $\min\{d(u,v) : u,v\in C \text{ distinct}\}$ is the \emph{minimum distance} of $C$. An important quantity in coding theory is the maximum size $A_q(n,d)$ of a code of length $n$ and minimum distance at least $d$:
\begin{equation*}
A_q(n,d) =\max\{|C| :  C\subseteq \q^n \text{ has minimum distance at least $d$}\}.
\end{equation*}
These codes are often called \emph{block codes} since they can be used to encode messages into fixed-length blocks of words from a code $C$. The most studied case is that of \emph{binary codes}, that is $q=2$ and $\q=\{0,1\}$. In this case $\q$ is suppressed from the notation. An excellent reference work on coding theory is MacWilliams, Sloane \cite{MacWilliamsSloane}.

Lower bounds for $A_q(n,d)$ are mostly obtained using explicit construction of codes. Our focus here is on upper bounds obtained using semidefinite programming.

\subsection{Lov\'asz' $\vartheta$ and Delsarte's linear programming bound}
Let $\Gamma_q(n,d)$ be the graph on vertex set $\q^n$, connecting two words $u,v\in \q^n$ if $d(u,v)<d$. Then the codes of minimum distance at most $d$ correspond exactly to the independent sets in $\Gamma_q(n,d)$ and hence 
\begin{equation*}
A_q(n,d)=\alpha(\Gamma_q(n,d)),
\end{equation*}
where $\alpha$ denotes the independence number of a graph. The graph parameter $\vartheta'$, defined in Schrijver \cite{Schrijver2}, is a slight modification of the Lov\'asz theta number discussed in more detail in Section \ref{Spherical codes}. It can be defined by the following semidefinite program
\begin{equation*}
\vartheta'(\Gamma) =\max\{\prodtrace{X}{J} : \prodtrace{X}{I}=1, X_{uv}=0 \text{ if $uv\in E(\Gamma)$}, X\geq 0, X\succeq 0\},
\end{equation*} 
and it is a by-now classical lemma of Lov\'asz \cite{Lovasz} that for any graph $\Gamma$ the inequality $\alpha(\Gamma)\leq \vartheta'(\Gamma)\leq \vartheta(\Gamma)$ holds.

Hence $\vartheta'(\Gamma_q(n,d))$ is an (exponential size) semidefinite programming upper bound for $A_q(n,d)$. Using symmetry reduction, this program can be solved efficiently.  For this, observe that the SDP is invariant under the action of $\Aut(q,n)$. Hence we may restrict $X$ to belong to the matrix $*$-algebra $\MA$ of invariant matrices (the Bose-Mesner algebra of the Hamming scheme). The zero-one matrices $A_0,\ldots, A_n$ corresponding to the orbits of pairs:
\begin{equation*}
(A_i)_{u,v}=\begin{cases}1&\text{if $d(u,v)=i$},\\0&\text{otherwise},\end{cases}
\end{equation*}
form the canonical basis of $\MA$. Writing $X =\frac{1}{q^n}(x_0A_0+\cdots+x_nA_n)$, we obtain an SDP in $n+1$ variables:
\begin{align*}
\vartheta'(\Gamma_q(n,d))=\max\Big\{\sum_{i=0}^nx_i\tbinom{n}{i} :  x_0=1, &x_1=\cdots=x_{d-1}=0,\nonumber\\
&x_d,\ldots,x_n\geq 0,\  \sum_{i=0}^nx_iA_i\succeq 0\Big\}.
\end{align*}
The second step is to block diagonalize the algebra $\MA$. In this case $\MA$ is commutative, which means that the $A_i$ can be simultaneously diagonalized, reducing the positive semidefinite constraint to the linear constraints
\begin{equation*}
\sum_{i=0}^n x_iP_i(j)\geq 0\text{ for $j=0,\ldots,n$},
\end{equation*}  
in the eigenvalues $P_i(j)$ of the matrices $A_i$. The $P_i$ are the \emph{Krawtchouk polynomials} and are given by
\begin{equation*}
P_i(x)=\sum_{k=0}^i (-1)^k\tbinom{x}{k}\tbinom{n-x}{i-k}(q-1)^{i-k}.
\end{equation*}
Thus the semidefinite program is reduced to a linear program, which is precisely the \emph{Delsarte bound}  \cite{Delsarte1}. This relation between $\vartheta'$ and the Delsarte bound was recognized independently in McEliece, Rodemich, Rumsey \cite{McElieceRodemichRumsey} and Schrijver \cite{Schrijver1}. In the more general setting of (commutative) association schemes, Delsarte's linear programming bound is obtained from a semidefinite program by symmetry reduction and diagonalizing the Bose-Mesner algebra.

\subsection{Stronger bounds through triples}
Delsarte's bound is based on the distance distribution of a code, that is, on the number of code word-pairs for each orbit of pairs. Using orbits of triples, the bound can be tightened as was shown in Schrijver \cite{Schrijver2}. We describe this method for the binary case ($q=2$).

Denote by $\Stab(0,\Aut(2,n))\subseteq \Aut(2,n)$ the stabilizer of $0$. That is, $\Stab(0,\Aut(2,n))$ consists of just the permutations of the $n$ positions.
 
For any $u,v,w\in \{0,1\}^n$, denote by $O(u,v,w)$ the orbit of the triple $(u,v,w)$ under the action of $\Aut(2,n)$. Observe that since $\Aut(2,n)$ acts transitively on $\{0,1\}^n$, the orbits of triples are in bijection with the orbits of pairs under $\Stab(0,\Aut(2,n))$, since we may assume that $u$ is mapped to $0$. There are $\tbinom{n+3}{3}$ such orbits, indexed by integers $0\leq t\leq i,j\leq n$. Indeed, the orbit of $(v,w)$ under $\Stab(0,\Aut(2,n))$ is determined by the sizes $i,j$ and $t$ of respectively the supports of $v$ and $w$ and their intersection. 

Let $C\subseteq \{0,1\}^n$ be a code. We denote by $\lambda_{i,j}^t$ the number of triples $(u,v,w)\in C^3$ in the orbit indexed by $i,j$ and $t$. This is the analogue for triples of the distance distribution.

Denote by $M_C$ the zero-one $\{0,1\}^n\times\{0,1\}^n$ matrix defined by $(M_C)_{u,v}=1$ if and only if $u,v\in C$. So $M_C$ is a rank 1 positive semidefinite matrix. Now consider the following two matrices
\begin{equation*}
M' =\sum_{\sigma\in \Aut(2,n), 0\in\sigma C} M_{\sigma C},\qquad M'' =\sum_{\sigma\in\Aut(2,n), 0\not\in \sigma C}M_{\sigma C}.
\end{equation*}
Clearly, $M'$ and $M''$ are nonnegative and positive semidefinite. Furthermore, by construction $M'$ and $M''$ are invariant under the action of the stabilizer $\Stab(0,\Aut(2,n))$. Equivalently, any entry $M'_{uv}$ (or $M''_{uv}$) only depends on the orbit of $(0,u,v)$ under the action of $\Aut(2,n)$. In fact, $M'_{uv}$ equals (up to a factor depending on the orbit) the number of triples $(c,c',c'')\in C^3$ in that orbit.

The matrix $*$-algebra $\MA_n$ of complex $\{0,1\}^n\times\{0,1\}^n$ matrices invariant under $\Stab(0,\Aut(2,n))$ has a basis of zero-one matrices corresponding to the orbits of pairs under the action of $\Stab(0,\Aut(2,n))$: $\{A_{i,j}^t :  0\leq t\leq i,j\leq n\}$. This algebra is called the \emph{Terwilliger algebra} of the (binary) Hamming scheme. The facts about $M'$ and $M''$ above lead to a semidefinite programming bound for $A(n,d)$, in terms of variables $x_{i,j}^t$. Nonnegativity of $M',M''$ translates into nonnegativity of the $x_{i,j}^t$. Excluding positive distances smaller than $d$ translates into setting $x_{i,j}^t=0$ for orbits containing two words at distance $1$ through $d-1$. Semidefiniteness of $M'$ and $M''$ translate into
\begin{equation*}
\sum_{i,j,t}x_{i,j}^tA_{i,j}^t\succeq 0,\qquad \sum_{i,j,t}(x_{i+j-t,0}^0-x_{i,j}^t)A_{i,j}^t\succeq 0.
\end{equation*}
There are some more constraints. In particular, for $u,v,w\in \{0,1\}^n$ the variables of the orbits of $(u,v,w)$ and $(w,u,v)$ (or any other reordering) must be equal. This further reduces the number of variables. For more details, see Schrijver \cite{Schrijver2}.

Having reduced the SDP-variables by using the problem symmetry, the second step is to replace the positive semidefinite constraints by equivalent conditions on smaller matrices using a block diagonalization of $\MA_n$. The block diagonalization of $\MA_n$ is given explicitly in Schrijver \cite{Schrijver2} (see also Vallentin \cite{Vallentin2},  Dunkl \cite{Dunkl}). The number of blocks equals $1+\lfloor\frac{n}{2}\rfloor$, with sizes $n+1-2k$ for $k=0,\ldots,\lfloor\frac{n}{2}\rfloor$. Observe that the sum of the squares of these numbers indeed equals $\tbinom{n+3}{3}$, the dimension of $\MA_n$.   

For the non-binary case, an SDP upper bound for $A_q(n,d)$ can be defined similarly. The main differences are, that the orbits or triples are now indexed by four parameters $i,j,p,t$ satisfying $0\leq p\leq t\leq i,j$, $i+j-t\leq n$. Hence the corresponding invariant algebra is a matrix $*$-algebra of dimension $\tbinom{n+4}{4}$. In that case the block diagonalization is a bit more involved having blocks indexed by the integers $a,k$ with $0\leq a\leq k\leq n+a-k$, and size $n+a+1-2k$. See Gijswijt, Schrijver, Tanaka \cite{GijswijtSchrijverTanaka} for details.

\subsection{Hierarchies and $k$-tuples}

\subsubsection*{Moment matrices}
Let $V$ be a finite set and denote by $\Pow(V)$ its power set. For any $y:\Pow(V)\to \C$, denote by $M(y)$ the $\Pow(V)\times \Pow(V)$
matrix given by
\begin{equation*}
[M(y)]_{S,T} =y_{S\cup T}.
\end{equation*}
Such a matrix is called a \emph{combinatorial moment matrix} (see Laurent \cite{LaurentSurvey}). 

Let $C\subseteq V$ and denote by $x =\chi^C\in \{0,1\}^V$ the zero-one incidence vector of $C$. The vector $x$ can be ``lifted'' to $y\in \{0,1\}^{\Pow(V)}$ by setting $y_S =\prod_{v\in S} x_v$ for all $S\in \Pow(V)$. Note that $y_\emptyset=1$. Then the moment matrix $M(y)$ is positive semidefinite since $M(y)=yy^{\transp}$. 
This observation has an important converse (see Laurent \cite{LaurentSurvey}).
\begin{theorem}\label{thm:moment}
Let $M(y)$ be a moment matrix with $y_\emptyset=1$ and $M(y)\succeq 0$. Then $M(y)$ is a convex combination of moment matrices $M(y_1),\ldots, M(y_k)$, where $y_i$ is obtained from lifting a vector in $\{0,1\}^V$.
\end{theorem}

Let $\Gamma=(V,E)$ be a graph on $V$. Theorem \ref{thm:moment} shows that the independence number is given by
\begin{equation}
\label{alphamoment}
\begin{split}
\alpha(\Gamma)=\max\Big\{\sum_{v\in V}y_{\{v\}} : & \; y_{\emptyset}=1, y_{S}=0 \text{ if $S$ contains an edge}, \\
& \; M(y)\succeq 0\Big\}.
\end{split}
\end{equation}
By replacing $M(y)$ by suitable principal submatrices, this (exponential size) semidefinite program can be relaxed to obtain more tractable upper bounds on $\alpha(\Gamma)$. For example, restricting rows and columns to sets of size at most 1 (and restricting $y$ to sets of size $\leq 2$), we obtain the Lov\'asz $\vartheta$ number (adding nonnegativity of $y_S$ for $S$ of size 2 gives $\vartheta'$).   

To describe a large class of useful submatrices, fix a nonnegative integer $s$ and a set $T\subseteq V$. Define $M_{s,T}=M_{s,T}(y)$ to be the matrix with rows and columns indexed by sets $S$ of size at most $s$, defined by 
\begin{equation}\label{MsT1}
[M_{s,T}]_{S,S'} =y_{S\cup S'\cup T}.
\end{equation}
So $M_{s,T}$ is a principal submatrix of $M(y)$, except that possibly some rows and columns are duplicated. For fixed $k$, we can restrict $y$ to subsets of size at most $k$ and replace $M(y)\succeq 0$ in (\ref{alphamoment}) by conditions
\begin{equation*}
M_{s,T}\succeq 0\quad\text{for all $T$ of size at most $t$},
\end{equation*}
where $2s+t\leq k$. This yields upper bounds on $\alpha(\Gamma)$ defined in terms of subsets of $V$ of size at most $k$.
The constraints (\ref{MsT1}) can be strengthened further to
\begin{equation}\label{MsT2}
\tilde{M}_{s,T} =\begin{pmatrix}M_{s,T'\cup T''}\end{pmatrix}_{T',T''\subseteq T}\succeq 0\quad \text{for all $T$ of size $t$}.
\end{equation}
Here $\tilde{M}_{s,T}$ is a $2^{|T|}\times 2^{|T|}$ matrix of smaller matrices, which can be seen as a submatrix of $M(y)$ (again up to duplication of rows and columns). By the moment-structure of $\tilde{M}_{s,T}$, condition (\ref{MsT2}) is equivalent to 
\begin{equation*}
\sum_{R'\subseteq R} (-1)^{|R\setminus R'|} M_{s,R'}\succeq 0\quad\text{for all $R$ of size at most $t$}.
\end{equation*}

Using SDP constraints of the form (\ref{MsT1}) and (\ref{MsT2}) with $2s+t\leq k$, we can obtain the (modified) theta number ($k=2$), Schrijver's bound using triples ($k=3$), or rather a strengthened version from Laurent \cite{LaurentTriples}, and the bound from Gijswijt, Mittelmann, Schrijver \cite{GijswijtMittelmannSchrijver} ($k=4$).

Several hierarchies of upper bounds have been proposed. The hierarchy introduced by Lasserre \cite{Lasserre} (also see Laurent \cite{LaurentTriples}) is obtained using (\ref{MsT1}) by fixing $t=0$ and letting $s=1,2,\ldots $ run over the positive integers. The hierarchy described in Gijswijt, Mittelmann, Schrijver \cite{GijswijtMittelmannSchrijver} is obtained using (\ref{MsT1}) by letting $k=1,2,\ldots$ run over the positive integers and restricting $s$ and $t$ by $2s+t\leq k$. Finally, the hierarchy defined in Gvozdenovi\'c, Laurent, Vallentin \cite{GvozdenovicLaurentVallentin} fixes $s=1$ and considers the constraints (\ref{MsT2}), where $t=0,1,\ldots$ runs over the nonnegative integers. There, also computational results are obtained for Paley graphs in the case $t=1,2$.

A slight variation is obtained by restricting $y$ to sets of at least one element, deleting the row and column corresponding to the empty set from (submatrices of) $M(y)$. The normalization $y_\emptyset=1$ is then replaced by $\sum_{v\in V}y_{\{v\}}=1$ and the objective is replaced by $\sum_{u,v\in V} y_{\{u,v\}}$. Since $M_{1,\emptyset}\succeq 0$ implies that $(\sum_{v\in V} y_{\{v\}})^2\leq y_{\emptyset}\sum_{u,v\in V} y_{\{u,v\}}$, this gives a relaxation. This variation was used in Schrijver \cite{Schrijver2} and Gijswijt, Schrijver, Tanaka \cite{GijswijtSchrijverTanaka}.

\subsubsection*{Symmetry reduction}
Application to the graph $\Gamma=\Gamma_q(n,d)$ with vertex set $V=\q^n$, gives semidefinite programming bounds for $A_q(n,d)$. A priory, this gives an SDP that has exponential size in $n$. However, using the symmetries of the Hamming space, it can be reduced to polynomial size as follows.

The action of $\Aut(q,n)$ on the Hamming space $V$, extends to an action on $\Pow(V)$ and hence on vectors and matrices indexed by $\Pow(V)$. Since the semidefinite program in (\ref{alphamoment}) is invariant under $\Aut(q,n)$, we may restrict $y$ (and hence $M(y)$) to be invariant under $\Aut(q,n)$, without changing the optimum value. The same holds for any of the bounds obtained using principal submatrices. Thus $M(y)_{S,T}$ only depends on the orbit of $S\cup T$. In particular, for any fixed integer $k$, the matrices $M_{s,T}$ with $2s+|T|\leq k$ are described using variables corresponding to the orbits of sets of size at most $k$. Two matrices $M_{s,T}$ and $M_{s,T'}$ are equal (up to permutation of rows and columns) when $T$ and $T'$ are in the same orbit under $\Aut(q,n)$. Hence, the number of variables is bounded by the number of orbits of sets of size at most $k$ ($O(n^{2^{k-1}-1})$ in the binary case), and for each such orbit there are at most a constant number of distinct constraints of the form (\ref{MsT1},\ref{MsT2}). This concludes the first step in the symmetry reduction.    

The matrices $M_{s,T}$ are invariant under the subgroup of $\Stab(T,\Aut(q,n))$ that fixes each of the elements of $T$. The dimension of the matrix $*$-algebra of invariant matrices is polynomial in $n$ and hence the size of the matrices can be reduced to polynomial size using the regular $*$-representation.  

By introducing more duplicate rows and columns and disregarding the row and column indexed by the empty set, we may index the matrices $M_{s,T}$ by \emph{ordered} $s$-tuples and view $M_{s,T}$ as an $(\q^n)^s \times (\q^n)^s$ matrix. Let $\MA_{s,T}$ be the matrix $*$-algebra of $(\q^n)^s \times (\q^n)^s$  matrices invariant under the the subgroup of $\Stab(T,\Aut(q,n))$. This matrix $*$-algebra can be block diagonalized using techniques from Gijswijt \cite{Gijswijt}.

\subsection{Generalizations}
The semidefinite programs discussed in the previous sections, are based on the distribution of pairs, triples and $k$-tuples in a code. For each tuple (up to isometry) there is a variable that reflects the number of occurrences of that tuple in the code. Exclusion of pairs at distance $1,\ldots, d-1$, is modeled by setting variables for violating tuples to zero. 

Bounds for other types of codes in the Hamming space can be obtained by setting the variables for excluded tuples to zero. This does not affect the underlying algebra of invariant matrices or the symmetry reduction. For triples in the binary Hamming space, this method was applied to orthogonality graphs in de Klerk, Pasechnik \cite{deKlerkPasechnik2} and to pseudo-distances in Bachoc, Z\'emor \cite{BachocZemor}. Lower bounds on (nonbinary) covering codes have been obtained in Gijswijt \cite{GijswijtThesis} by introducing additional linear constraints.

Bounds on constant weight binary codes were given in Schrijver \cite{Schrijver2}.

In the nonbinary Hamming space, the symbols $0,\ldots, q-1$ are all interchangeable. In the case of Lee-codes, the dihedral group $D_q$ acts on the alphabet, which leads to a smaller isometry group $D_q\wr S_n$. The Bose-Mesner algebra of the Lee-scheme is still commutative. The corresponding linear programming bound was studied in Astola \cite{Astola}. To the best knowledge of the author, stronger bounds using triples have not been studied in this case.

\section{Crossing numbers}
\label{Crossing numbers}

We describe an application of the regular $*$-representation to
obtain a lower bound on the crossing number of complete bipartite
graphs.
This was described in de Klerk, Pasechnik, Schrijver \cite{deKlerkPasechnikSchrijver}, and
extends a method of de Klerk, et. al. \cite{deKlerkMaharry}.

The {\em crossing number} $\kr(\Gamma)$ of a graph $\Gamma$ is the minimum number
of intersections of edges when $\Gamma$ is drawn in the
plane such that all vertices are distinct.
The {\em complete bipartite graph} $K_{m,n}$ is the graph with
vertices $1,\ldots,m,u_1,\ldots,u_n$ and edges all pairs
$iu_j$ for $i\in[m]$ and $j\in[n]$.
(This notation will be convenient for our purposes.)

This relates to the problem raised by the paper of
Zarankiewicz \cite{Zarankiewicz}, asking if
\begin{equation}\label{8jl10a}
\kr(K_{m,n})\stackrel{?}{=}
Z(m,n) =
\lfloor\kfrac14{(m-1)^2}\rfloor
\lfloor\kfrac14{(n-1)^2}\rfloor.
\end{equation}
In fact, Zarankiewicz claimed to have a proof, which however
was shown to be incorrect.
In (\ref{8jl10a}), $\leq$ follows from a direct construction.
Equality was proved by
Kleitman \cite{Kleitman} if $\min\{m,n\}\leq 6$ and by
Woodall \cite{Woodall} if $m\in\{7,8\}$ and $n\in\{7,8,9,10\}$.

Consider  any $m,n$.
Let $Z_m$ be the set of cyclic permutations of $[m]$
(that is, the permutations with precisely one orbit).
For any drawing of $K_{m,n}$ in the plane and for any $u_i$,
let $\gamma(u_i)$ be the cyclic permutation $(1,i_2,\ldots,i_m)$
such that the edges leaving $u_i$ in clockwise order, go to
$1,i_2,\ldots,i_m$ respectively.

For $\sigma,\tau\in Z_m$, let
$C_{\sigma,\tau}$ be equal to the minimum number of crossings
when we draw $K_{m,2}$ in the plane such that $\gamma(u_1)=\sigma$
and $\gamma(u_2)=\tau$.
This gives a matrix $C=(C_{\sigma,\tau})$ in $\oR^{Z_m\times Z_m}$.
Then the number $\alpha_m$ is defined by:
\begin{equation}\label{13ja05a}
\alpha_m =\min\{\langle X, C \rangle 
:
X\in\oR_+^{Z_m\times Z_m},
X \succeq 0,
\langle X, J \rangle =1\},
\end{equation}
where $J$ is the all-one matrix in $\oR^{Z_m\times Z_m}$.

Then $\alpha_m$ gives a lower bound on $\kr(K_{m,n})$,
as was shown by de Klerk, et. al. \cite{deKlerkMaharry}.

\begin{mytheorem}
$\displaystyle
\kr(K_{m,n})\geq \kfrac12n^2\alpha_m
-
\kfrac12n \lfloor\kfrac14{(m-1)^2}\rfloor
$
for all $m,n$.
\end{mytheorem}

\begin{proof}
Consider a drawing of $K_{m,n}$ in the plane with
$\kr(K_{m,n})$ crossings.
For each cyclic permutation $\sigma$, let $d_{\sigma}$ be the
number of vertices $u_i$ with $\gamma(u_i)=\sigma$.
Consider $d$ as column vector in $\oR^{Z_m}$, and define the matrix
$X$ in $\oR^{Z_m\times Z_m}$ by
$$
X =n^{-2}dd^{\sf T}.
$$
Then $X$ satisfies the constraints in (\ref{13ja05a}), hence
$\alpha_m\leq \langle X, C \rangle$.
For $i,j=1,\ldots,n$, let $\beta_{i,j}$ denote the number
of crossings of the edges leaving $u_i$ with the edges leaving
$u_j$.
So if $i\neq j$, then
$\beta_{i,j}\geq C_{\gamma(u_i),\gamma(u_j)}$.
Hence
$$
n^2\langle X, C \rangle=
\langle dd^{\sf T} C \rangle)=
d^{\sf T}Cd=
\sum_{i,j=1}^n
C_{\gamma(u_i),\gamma(u_j)}
$$
$$
\leq
\sum_{\stackrel{i,j=1}{i\neq j}}^n
\beta_{i,j}
+
\sum_{i=1}^n
C_{\gamma(u_i),\gamma(u_i)}
=
2\kr(K_{m,n})+
n
\lfloor\kfrac14{(m-1)^2}\rfloor.
$$
Therefore,
$$
\kr(K_{m,n})
\geq
\kfrac12n^2\langle X, C \rangle 
-
\kfrac12n
\lfloor\kfrac14{(m-1)^2}\rfloor
\geq
\kfrac12 \alpha_m n^2
-
\kfrac12n
\lfloor\kfrac14{(m-1)^2}\rfloor.~~~~
\Box
$$
\end{proof}

The semidefinite programming problem (\ref{13ja05a}) defining $\alpha_m$ has
order $(m-1)!$.
Using the regular $*$-representation, we can reduce the size.
Fix $m\in\oN$.
Let $G =S_m\times \{-1,+1\}$, and define $h:G\to S_{Z_m}$ by
$$
h_{\pi,i}(\sigma) =\pi\sigma^i\pi^{-1}
$$
for $\pi\in S_m$, $i\in\{-1,+1\}$, $\sigma\in Z_m$.
So $G$ acts on $Z_m$.
It is easy to see that $C$ and (trivially) $J$ are $G$-invariant.

Using the regular $*$-representation, $\alpha_m$ up to $m=9$ was computed.
For $m=8$, $\alpha_8=5.8599856444\ldots$, implying
$$
\kr(K_{8,n})\geq 2.9299n^2-6n.
$$
This implies for each fixed $m\geq 8$, with an averaging argument over all
subgraphs $K_{8,n}$:
$$
\lim_{n\to\infty}\frac{\kr(K_{m,n})}{Z(m,n)}\geq 0.8371\frac{m}{m-1}.
$$

Moreover, $\alpha_9=7.7352126\ldots$, implying
$$
\kr(K_{9,n})\geq 3.8676063 n^2-8n,
$$
and for each fixed $m\geq 9$:
$$
\lim_{n\to\infty}\frac{\kr(K_{m,n})}{Z(m,n)}\geq 0.8594\frac{m}{m-1}.
$$
Thus we have an asymptotic approximation to Zarankiewicz's problem.

The orbits of the action of $G$ onto $Z_m\times Z_m$ can be identified as
follows.
Each orbit contains an element $(\sigma_0,\tau)$ with
$\sigma_0 =(1,\ldots,m)$.
So there are at most $(m-1)!$ orbits.
Next, $(\sigma_0,\tau)$ and $(\sigma_0,\tau')$ belong to the same
orbit if and only if $\tau'=\tau^g$ for some $g\in G$ that fixes $\sigma_0$.
(There are only few of such $g$.)
In this way, the orbits can be identified by computer, for $m$ not too large.
The corresponding values $C_{\sigma,\tau}$ for each orbit $[\sigma,\tau]$,
and the multiplication parameters, also can be found using elementary
combinatorial algorithms.
The computer does this all within a few minutes, for $m\leq 9$.
But the resulting semidefinite programming problem took, in 2006, seven days.
It was the largest SDP problem solved by then.

\section{Spherical codes}
\label{Spherical codes}

Finding upper bounds for codes on the sphere is our next application. Here one deals with infinite-dimensional  semidefinite programs  which are invariant under the orthogonal group which is continuous and compact. So we are in the situation for which the techniques of Section~\ref{Representation theory} work. 

We start with some definitions: The unit sphere $S^{n-1}$ of the Euclidean space $\R^n$
equipped with its standard inner product 
$x\cdot y=\sum_{i=1}^n x_iy_i$, is defined as usual by:
\begin{equation*}
S^{n-1}=\{(x_1,\dots,x_n)\in \R^n : x\cdot x=1\}.
\end{equation*}
It is a  compact space, on which the orthogonal group $O_n(\R)$ acts
homogeneously. The stabilizer $\Stab(x_0,O_n(\R))$ of one point $x_0\in
S^{n-1}$ can be identified with the orthogonal group of the orthogonal
complement $(\R x_0)^{\perp}$ of the line $\R x_0$ thus with
$O_{n-1}(\R)$, leading to an identification between the sphere and the
quotient space $O_n(\R)/O_{n-1}(\R)$. The unit sphere  is endowed with the standard
$O_n(\R)$-invariant Lebesgue measure $\mu$ with the normalization $\mu(S^{n-1})=1$.
The {\em angular distance} $d_{\theta}(x,y)$ of $(x,y)\in (S^{n-1})^2$ is
defined by
\begin{equation*}
d_{\theta}(x,y) =\arccos(x\cdot y)
\end{equation*}
and is $O_n(\R)$-invariant. Moreover,   the metric space $(S^{n-1},d_{\theta})$  is {\em two-point homogeneous} (see Section
\ref{Ex: commutative}) under
$O_n(\R)$.

The {\em minimal angular distance} $d_{\theta}(C)$ of a subset
$C\subset S^{n-1}$ is by definition 
the minimal  angular distance of pairs of distinct
elements of $C$. It is a classical problem to determine the maximal
size of $C$, subject to the condition that $d_{\theta}(C)$ is greater
or equal to some given minimal value $\theta_{\min}$. 
This problem is the fundamental question of the theory of error
correcting codes (see e.g. Conway, Sloane \cite{ConwaySloane}, Ericson, Zinoviev \cite{EricsonZinoviev}). In this context, the subsets $C\subset
S^{n-1}$ are referred to as {\em spherical codes}. In geometry, the case
$\theta_{\min}=\pi/3$ corresponds to the famous {\em kissing number problem} which asks
for the maximal number of spheres that can simultaneously 
touch a central sphere without overlapping, all spheres having the
same radius (see e.g. Conway, Sloane \cite{ConwaySloane}). 

We introduce notations in a slightly more general framework. We say
that $C\subset S^{n-1}$ {\em avoids $\Omega\subset (S^{n-1})^2$} if,
for all $(x,y)\in C^2$, $(x,y)\notin \Omega$. We define, for a measure $\lambda$,
\begin{equation*}
A(S^{n-1}, \Omega, \lambda) =\sup \big\{ \lambda(C): C\subset
S^{n-1} \text{ measurable},\ C\text{ avoids }\Omega\big\}.
\end{equation*}
We have in mind the following situations of interest:
\begin{enumerate}
\item[(i)] $\Omega=\{(x,y) : d_{\theta}(x,y)\in ]0,\theta_{\min}[\}$ and $\lambda$ is the counting
    measure denoted $\mu_c$. Then, the $\Omega$-avoiding sets are exactly the
    spherical codes $C$ such that $d_{\theta}(C)\geq \theta_{\min}$.
\item[(ii)] $\Omega=\{(x,y) : d_{\theta}(x,y)=\theta\}$ for some value
  $\theta\neq 0$, and $\lambda=\mu$. Here we consider subsets
  avoiding only one value of distance, so these subsets can be infinite.
This case has interesting connections with the famous problem of finding
the  {\em chromatic number of Euclidean space} (see Soifer \cite{Soifer}, Bachoc et. al. \cite{BachocNebeOliveiraVallentin}, Oliveira, Vallentin \cite{OliveiraVallentin}, Oliveira \cite{Oliveira})
\item[(iii)] $\Omega=\{(x,y) : d_{\theta}(x,y)\in ]0,\theta_{\min}[ \text{ or
      } (x,y) \notin \Scap(e,\phi)^2 \}$, and $\lambda=\mu_c$.
Here $\Scap(e,\phi)$ denotes the {\em spherical cap} with center $e$
and radius $\phi$:
\begin{equation*}
\Scap(e,\phi) =\big\{x\in S^{n-1} : d_{\theta}(e,x)\leq \phi\big\}
\end{equation*}
and we are dealing with subsets of a spherical cap with given minimal
angular distance.
\end{enumerate}

The computation of $A(S^{n-1}, \Omega, \lambda)$ is a difficult and unsolved problem
in general, so one aims
at finding lower and upper bounds for this number.
To that end, for finding upper bounds, we borrow ideas from combinatorial optimization. Indeed, if one
thinks of the pair $(S^{n-1},\Omega)$ being a graph with vertex set
$S^{n-1}$ and edge set $\Omega$, then a set $C$ avoiding $\Omega$
corresponds to an {\em independent set}, and
$A(S^{n-1}, \Omega, \lambda)$  to the {\em independence number} of
this graph. In general, for a graph $\Gamma$ with vertex set $V$ and edge set
$E$, an {\em independent set} is a subset $S$ of $V$ such that
no pairs of elements of $S$ are connected by an edge, and the {\em independence number}
$\alpha(\Gamma)$ of $\Gamma$ is the maximal number of elements of an
independent set. The  {\em Lov\'asz theta number}
$\vartheta(\Gamma)$ was introduced by Lov\'asz \cite{Lovasz}. It gives an upper bound of the independence number $\alpha(\Gamma)$ of a
graph $\Gamma$, and it is the optimal solution of a semidefinite
program. We review Lov\'asz theta number in Section
\ref{Theta}. Then, in Section \ref{Generalizations Theta},
we introduce generalizations of this notion, in the form of conic
linear programs, that provide upper bounds for $A(S^{n-1}, \Omega,
\lambda)$. Using symmetry reduction, in the cases (i), (ii), (iii) above,  it is possible to approximate
these conic linear programs with semidefinite programs, that can be
practically computed when the dimension $n$ is not too large. This step
is explained in Section \ref{Reduction}. It involves the description 
of the cones
of $G$-invariant positive definite functions, for the
groups
$G=O_n(\R)$ and $G=\Stab(x_0,O_n(\R))$.
We provide this description, following the lines of Section \ref{Representation theory},
in Sections \ref{Pdf full} and \ref{Pdf stab}.
Finally, in Section \ref{Strengthening} we indicate further applications.

\subsection{Lov\'asz $\vartheta$}\label{Theta} This number can be defined in   many
equivalent ways (see Lov\'asz \cite{Lovasz},  Knuth \cite{Knuth}). We present here the most appropriate one  in view of
our purpose, which is the generalization to $S^{n-1}$. 
\begin{mydefinition}
The {\em theta number} of the graph $\Gamma=(V,E)$ with vertex set
$V=\{1,2,\dots,n\}$  is 
defined by
\begin{equation}\label{theta primal}
  \vartheta(\Gamma) = \max\big\{  \langle X, J_n\rangle : X \succeq 0,
  \langle X, I_n\rangle =1, X_{ij}=0 \text{ for all } (i,j)\in E \big\}
\end{equation}
where $I_n$ and $J_n$ denote respectively the identity and the
all-one matrices of size $n$.
\end{mydefinition}

The dual program gives another expression of $\vartheta(\Gamma)$
(there is no {\em duality gap} here
because $X=I_n$ is a strictly feasible solution of \eqref{theta
  primal} so the Slater condition is fulfilled):
\begin{equation}\label{theta dual}
  \vartheta(\Gamma) = \min \big\{  t: X \succeq 0,
  X_{ii}=t-1,  X_{ij}=-1 \text{ for all } (i,j)\notin E \big\}.
\end{equation}
We have already mentioned that this number provides an upper bound
for 
the independence number
$\alpha(\Gamma)$ of the graph $\Gamma$. It also provides a lower bound for the
{\em chromatic number} $\chi(\overline{\Gamma})$ of the complementary graph
$\overline{\Gamma}$ (the chromatic number of a
graph is the minimal number of colors needed to color its vertices  so
that two 
connected vertices receive different colors);
this is the content of the celebrated {\em Sandwich theorem} of Lov\'asz \cite{Lovasz}:
\begin{mytheorem}
\begin{equation}\label{Sandwich}
\alpha(\Gamma)\leq \vartheta(\Gamma)\leq \chi(\overline{\Gamma})
\end{equation}
\end{mytheorem}

By modifying the definition of the theta number we get the strengthening $\alpha(\Gamma)\leq \vartheta'(\Gamma) \leq \vartheta(\Gamma)$, where
\begin{equation}\label{theta prime primal}
\begin{array}{ll}
  \vartheta'(\Gamma) = \max\big\{  \langle X, J_n\rangle : & X \succeq
  0,\  X\geq 0,\\
  & \langle X, I_n\rangle =1, \ X_{ij}=0 \text{ for all } (i,j)\in E
  \big\}.
\end{array}
\end{equation}
Again, this program has an equivalent dual form:
\begin{equation}\label{theta prime dual}
\begin{array}{ll}
  \vartheta'(\Gamma) = \min \big\{  t: X \succeq 0,
  & X_{ii}\leq t-1,\\
& X_{ij}\leq -1 \text{ for all } (i,j)\notin E \big\}.
\end{array}
\end{equation}

\subsection{Generalizations of Lov\'asz $\vartheta$ to the  sphere}\label{Generalizations Theta}
In order to obtain the wanted generalization, it is natural to
replace in \eqref{theta prime primal} 
the cone of positive semidefinite matrices indexed by the vertex set
$V$ of the underlying graph, with the cone of
continuous, positive definite functions on the sphere, defined in
Section \ref{Pdf}. In fact, there
is a small difficulty here due to the fact that, in contrast with a
finite dimensional Euclidean space, the space $\CC(X)$ of continuous
functions on an infinite compact space cannot be identified with its
topological dual. Indeed, the {\em topological dual} $\CC(X)^*$ of
$\CC(X)$ (i.e. the space
of continuous linear forms on $\CC(X)$) 
is the space $\MM(X)$ of complex valued Borel regular measures on $X$
(see e.g. Rudin \cite[Theorem 6.19]{Rudin}, Conway \cite[Appendix C]{Conway}).
Consequently, the two forms of $\vartheta'(\Gamma)$ given by
\eqref{theta prime primal} and \eqref{theta prime dual} lead in the infinite case
to two pairs of conic linear programs. As we shall see, the right
choice between the two in order to obtain an appropriate bound for 
$A(S^{n-1}, \Omega, \lambda)$ depends on the set $\Omega$.

\begin{mydefinition} Let $\Omega^c =\{(x,y) : (x,y)\notin \Omega \text{ and
    }x\neq y\}$. Let
\begin{equation}\label{theta1}
\begin{array}{ll}
\vartheta_1(S^{n-1}, \Omega) =\inf\big\{\  t : F\succeq 0, & F(x,x)\leq t-1, \\
& F(x,y)\leq -1  \text{ for all }(x,y) \in \Omega^c\ \big\},
\end{array}
\end{equation}
and
\begin{equation}\label{theta2}
\begin{array}{ll}
\vartheta_2(S^{n-1},\Omega) =\sup\big\{ \ \langle F, 1\rangle : & F\succeq 0, \ F\geq 0, \\
& \langle F, \1_{\Delta}\rangle =1,\\
& F(x,y)=0   \text{ for all }(x,y)\in \Omega\ \big\}.
\end{array}
\end{equation}
\end{mydefinition}
In the above programs, $F$ belongs to $\CC((S^{n-1})^2)$. The
  notation $F\geq 0$ stands for ``$F$ takes nonnegative values''. The
function taking the constant value $1$ is denoted $1$, so that 
$\langle F, 1\rangle$ equals the integral of $F$ over $(S^{n-1})^2$. 
In contrast, with $\langle F, \1_{\Delta}\rangle$ we mean the integral
of the one variable function $F(x,x)$ over $\Delta =\{(x,x): x\in
S^{n-1}\}$. Let us notice that, since $F$ is continuous in both programs, the sets
$\Omega$ and $\Omega^c$ can be replaced by their topological closures 
$\overline{\Omega}$ and $\overline{\Omega^c}$. 

A {\em positive semidefinite measure} on $(S^{n-1})^2$ is one that satisfies 
$\langle \lambda, f\rangle\geq 0$ for all $f\succeq 0$, where 
\begin{equation*}
\langle \lambda, f\rangle =\int_{X} f(x)d\lambda(x).
\end{equation*}
and this property
is denoted $\lambda\succeq 0$. In a similar way a nonnegative measure
$\lambda$ is denoted $\lambda\geq 0$. 

\begin{mytheorem}\label{Th5} With the above notations and definitions, we have:
\begin{enumerate} 
\item  If $\overline{\Omega^c}\cap \Delta=\emptyset$, then the program
  dual to $\vartheta_1$ reads
\begin{equation*}
\begin{array}{ll}
\vartheta_1^*(S^{n-1},\Omega) =\sup\big\{ \ \langle \lambda, 1 \rangle : & \lambda\succeq 0, \ \lambda\geq 0, \\
& \lambda(\Delta) =1,\\
& \supp(\lambda)\subset   \overline{\Omega^c}\cup \Delta\ \big\}
\end{array}
\end{equation*}
and
\begin{equation}\label{theta1*}
A(S^{n-1},\Omega,\mu_c)\leq \vartheta_1(S^{n-1},\Omega)=\vartheta_1^*(S^{n-1},\Omega).
\end{equation}
\item If $\overline{\Omega}\cap \Delta=\emptyset$, then
the program  dual to $\vartheta_2$ reads
\begin{equation*}
\begin{array}{ll}
\vartheta_2^*(S^{n-1}, \Omega) =\inf\big\{\  t : &\lambda\succeq 0, \\
& \lambda\leq t\mu_{\Delta}-\mu^2 \text{ over } (S^{n-1})^2\setminus
\overline{\Omega}\  \big\},
\end{array}
\end{equation*}
and 
\begin{equation}\label{theta2*}
A(S^{n-1},\Omega,\mu)\leq \vartheta_2(S^{n-1},\Omega)=\vartheta_2^*(S^{n-1},\Omega).
\end{equation}
\end{enumerate}
\end{mytheorem}

\begin{proof} The dual programs are computed in a standard way (see Duffin \cite{Duffin}, Barvinok \cite{Barvinok}). In
  order to prove that there is no duality gap, one can apply the
  criterion \cite[Theorem 7.2]{Barvinok}. 

For the inequality \eqref{theta1*}, let $C$ be a maximal spherical
code 
avoiding $\Omega$. Then, the measure $\lambda =\delta_{C^2}/|C|$,
where $\delta$ denotes the Dirac measure, is a feasible solution of
$\vartheta_1^*$ with optimal value $|C|=A(S^{n-1},\Omega,\mu_c)$.

For the inequality \eqref{theta2*}, let $(t,\lambda)$ be a feasible
solution of $\vartheta_2^*$. Then, if $C$ avoids $\Omega$, $C^2\subset (S^{n-1})^2\setminus
\overline{\Omega}$. Thus, $0\leq \lambda(C^2)\leq t\mu(C)-\mu(C)^2$,
leading to the wanted inequality $\mu(C)\leq t$.
\smartqed\qed
\end{proof}

\begin{myexample}
In the situations  (i) and (ii) above, the
set $\Omega$ fulfills the condition 1.\ of Theorem \ref{Th5} while for (iii) 
we are in the case 2.
\end{myexample}

\subsection{Positive definite functions invariant under the full orthogonal
group}\label{Pdf full} It is a classical result of Schoenberg \cite{Schoenberg} that
these functions, in the variables $(x,y)\in (S^{n-1})^2$, are exactly the nonnegative linear combinations of {\em Gegenbauer
  polynomials} (Section \ref{Ex: commutative}) evaluated at the inner product $x\cdot y$. 
We briefly review this classical result, following
the lines of Section \ref{Representation theory}.

The space $\Hom_k^n$ of polynomials in $n$ variables $x_1,\dots,x_n$ which are
homogeneous of degree $k$ affords an action of the group $O_n(\R)$
acting linearly on the $n$ variables. The Laplace operator
$\Delta=\sum_{i=1}^n \frac{\partial^2}{\partial x_i^2}$ commutes with
this action, thus the subspace $\Harm_k^n$ defined by:
\begin{equation*}
\Harm_k^n =\{P\in \Hom_k^n : \Delta P=0\}
\end{equation*}
is a representation of $O_n(\R)$ which turns out to be irreducible (see e.g. Andrews, Askey, Roy \cite{AndrewsAskeyRoy}).
Going from polynomials to polynomial functions on the sphere, we
introduce $H_k^n$, the subspace of $\CC(S^{n-1})$ arising from elements
 of $\Harm_k^n$. Then, for $d\geq k$, $H_k^n$ is a subspace of the space
 $V_d =\Pold(S^{n-1})$ of polynomial functions on $S^{n-1}$ of degree up to $d$.
We have:
\begin{mytheorem} The irreducible $O_n(\R)$-decomposition of $\Pold(S^{n-1})$ is
  given by:
\begin{equation}\label{dec1}
\Pold(S^{n-1})=H_0^n\perp H_1^n\perp\dots \perp H_d^n,
\end{equation}
where, for all $k\geq 0$, $H_k^n\simeq \Harm_k^n$ is of dimension $h_k^n =\binom{n+k-1}{k}-\binom{n+k-3}{k-2}$.

The function
$Y_k(d_{\theta}(x,y))$ associated to $H_k^n$ following \eqref{pdf 2-pt hom}
equals:
\begin{equation}\label{zonal1}
Y_k(d_{\theta}(x,y))=h_k^n P_k^{((n-3)/2, (n-3)/2)}(x\cdot y).
\end{equation}
\end{mytheorem}

\begin{proof} For the sake of completeness we sketch a proof. Because $S^{n-1}$ is two-point homogeneous under
  $O_n(\R)$, the algebra $(V_d)^{(2)}$ related to the finite dimensional
  $O_n(\R)$-space $V_d$ is equal to
the set of polynomial functions in the
  variable $t =x\cdot y$ of degree up to $d$, thus has dimension
  $d+1$. On the other hand, it can be shown that the $d+1$ subspaces $H_k^n$, for
  $0\leq k\leq d$, are non
  zero and pairwise distinct. Then, the statement \eqref{dec1} 
  follows from the equality of the dimensions of  $(V_d)^{(2)}$ and of
  $\End^{O_n(\R)}(V_d)$ which are isomorphic algebras \eqref{isom}.
 It is worth to notice that we end up
  with a proof that the spaces $H_k^n$ are irreducible, without
  referring to the irreducibility of the spaces $\Harm_k^n$
(which can be proved in a similar way). The formula for
  $h_k^n$ follows by induction.

The functions  $Y_k(d_{\theta(x,y)})$ associated to the decomposition
\eqref{dec1} must be polynomials in the
variable $x\cdot y$ of degree $k$; let us denote them temporary by $Q_k(t)$. A  change of
variables shows that, for a function $f(x)\in \CC(S^{n-1})$ of the form
$f(x)=\varphi(x\cdot y)$ for some $y\in S^{n-1}$,
\begin{equation*}
\int_{S^{n-1}} f(x)d\mu(x)=c_n\int_{-1}^1 \varphi(t)(1-t^2)^{(n-3)/2}dt
\end{equation*}
for some constant $c_n$.
Then, because the subspaces $H_k^n$ are pairwise orthogonal, the polynomials $Q_k$ must satisfy the orthogonality conditions
\begin{equation*}
\int_{-1}^1 Q_k(t)Q_l(t)(1-t^2)^{(n-3)/2}dt=0
\end{equation*}
for all $k\neq l$ thus they must be equal to the Gegenbauer
polynomials up to a multiplicative factor. Integrating over
$S^{n-1}$ the 
formula \eqref{Ek} when $x=y$ shows that $Y_k(0)=h_k^n$ thus computes this factor.
\smartqed\qed
\end{proof}

\begin{mycorollary}  Let $P_k^n(t) =P_k^{((n-3)/2,
    (n-3)/2)}(t)$. Then, $F\in \CC(S^{n-1})^{O_n(\R)}_{\succeq 0}$ if and
  only if
\begin{equation}\label{Pdf sphere}
F(x,y)=\sum_{k\geq 0} f_k P_k^n(x\cdot y), \text{ with } f_k\geq 0 \text{ for all
}k\geq 0.
\end{equation}
\end{mycorollary}

\subsection{Positive definite functions invariant under  the
  stabilizer of one point}\label{Pdf stab}
We fix an element $e\in S^{n-1}$ and let $G=\Stab(e,O_n(\R))$. Then,
the orbit $O(x)$  of $x\in S^{n-1}$ under $G$ is the set
\begin{equation*}
O(x)=\{y\in S^{n-1} : e\cdot y=e\cdot x\}.
\end{equation*}
The orbit $O((x,y))$  of $(x,y)\in S^{n-1}$ equals
\begin{equation*}
O((x,y))=\{(z,t)\in (S^{n-1})^2 : (e\cdot  z, e\cdot t, z\cdot t)=(e\cdot  x, e\cdot y, x\cdot y)\}.
\end{equation*}
In other words, the orbits of $G$ on $(S^{n-1})^2$ are parametrized by
the triple of real
numbers $(e\cdot  x, e\cdot y, x\cdot y)$. So
the
$G$-invariant positive definite functions on $S^{n-1}$ are functions
of the three variables $u=e\cdot x$, $v=e\cdot y$, $t=x\cdot y$. Their
expression is computed in Bachoc, Vallentin \cite{BachocVallentin1}.

\begin{mytheorem}\label{Th dec2} The irreducible decomposition of $V_d=\Pold(S^{n-1})$
  under the action of $G =\Stab(e,S^{n-1})\simeq O_{n-1}(\R)$ is given by:

\begin{equation}\label{dec2}
\Pold(S^{n-1})\simeq \bigoplus_{k=0}^d  (\Harm_k^{n-1})^{d-k+1}.
\end{equation}
The coefficients of the matrix $Y_k=Y_k^n$ of size $d-k+1$ associated to $\Harm_k^{n-1}$
are equal to:
\begin{equation}\label{Pdf 2}
Y^n_{k,ij}(u,v,t)= u^iv^jQ_k^{n-1}(u,v,t)
\end{equation}
where $0\leq i,j\leq d-k$, and
\begin{equation*}
Q_k^{n-1}(u,v,t) =\big((1-u^2)(1-v^2)\big)^{k/2}P_k^{n-1}\Big(\frac{t-uv}{\sqrt{(1-u^2)(1-v^2)}}\Big)
\end{equation*}
\end{mytheorem}

\begin{proof} We refer to Bachoc, Vallentin \cite{BachocVallentin1} for the details. In order to
  obtain the $G$-decomposition of the space $\Pold(S^{n-1})$ we can
  start from the $O_n(\R)$-decomposition
\eqref{dec1}. The $O_n(\R)$-irreducible subspaces $H_k^n$ split into
smaller subspaces following the $O_{n-1}(\R)$-iso\-mor\-phisms (see e.g. Vilenkin, Klimyk \cite{VilenkinKlimyk}):
\begin{equation}
\Harm_k^n\simeq \bigoplus_{i=0}^k \Harm_i^{n-1}
\end{equation}
leading to \eqref{dec2}.
The expression \eqref{Pdf 2} follows from the definition \eqref{Ek} of
$E_{k,ij}(x,y)$
and from  the construction of
convenient basis $(e_{k,i,s})_s$. These basis arise   from explicit
embeddings of $H_{k}^{n-1}$ into $V_d$, defined as follows. For $x\in S^{n-1}$, let
$x=ue+\sqrt{1-u^2}\zeta$ where $\zeta\in (\R e)^{\perp}$ and
$\zeta\cdot \zeta=1$. Let $\varphi_{k,i}$ send $f\in H_k^{n-1}$ to
$\varphi_{k,i}(f)$  where
\begin{equation*}
\varphi_{k,i}(f)(x) =u^i (1-u^2)^{k/2} f(\zeta).
\end{equation*}
Then,  $\{e_{k,i,s} :  1\leq s\leq h_k^{n-1}\}$ is taken to be the image by
$\varphi_{k,i}$ of an
orthonormal basis of $H_k^{n-1}$.
\smartqed\qed
\end{proof}

\subsection{Reduction using symmetries}\label{Reduction}
If the set $\Omega$ is invariant under a subgroup $G$ of $O_n(\R)$,
then the feasible sets of the conic linear programs \eqref{theta1} and
\eqref{theta2} 
can be restricted to the $G$-invariant positive definite functions
$F$. Indeed, symmetry reduction of finite dimensional semidefinite programs extends to
infinite compact spaces, with the now familiar trick that replaces
the finite average over the elements of a finite group by the integral for the Haar measure of the group.
We apply this general principle to the
programs \eqref{theta1} and \eqref{theta2} and we focus on the cases
(i)--(iii) above.

\subsubsection{Case (i): $\Omega=\{(x,y) : d_{\theta}(x,y)\in
]0,\theta_{\min}[\}$.} 
The set $\Omega$ is invariant under $O_n(\R)$. According to Theorem \ref{Th5} we consider the
program \eqref{theta1}  where $F\succeq 0$ can be replaced by
\eqref{Pdf sphere}. The program $\vartheta_1$ becomes after a few simplifications:

\begin{equation}\label{Delsarte}
\begin{array}{ll}
\vartheta_1(S^{n-1}, \Omega) =\displaystyle \inf\big\{\  1+\sum_{k\geq 1} f_k :
&f_k\geq 0,  \\
&\displaystyle  1+\sum_{k\geq 1} f_k P_k^n(t) \leq 0  \quad t \in
         [-1,s] \big\},
\end{array}
\end{equation}
where $s =\cos(\theta_{\min})$. One recognises  in \eqref{Delsarte} the so-called {\em
  Delsarte linear programming bound} for the maximal number of
elements of a spherical code with minimal angular distance
$\theta_{\min}$, see Delsarte, Goethals, Seidel \cite{DelsarteGoethalsSeidel}, Kabatiansky, Levenshtein \cite{KabatianskyLevenshtein}, Conway, Sloane \cite[Chapter 9]{ConwaySloane}. The optimal
value of the linear program \eqref{Delsarte} is not known in general, although
explicit feasible solutions leading to very good bounds and also to
asymptotic bounds have been
constructed (Kabatiansky, Levenshtein \cite{KabatianskyLevenshtein}, Levenshtein \cite{Levenshtein}, Odlyzko, Sloane \cite{OdlyzkoSloane}). Moreover, this infinite dimensional linear program can be
efficiently approximated by semidefinite programs defined in the
following way: 
in order to deal
with only a finite number of variables $f_k$, one restricts to $k\leq
d$ (it amounts to restrict to $G$-invariant positive definite
functions of $V_d^{(2)}$). Then the polynomial $1+\sum_{k=1}^d
f_kP_k^n$ is required in \eqref{Delsarte} to be 
nonpositive over a certain  interval of real
numbers. By the theorem of Luk\'acs concerning non-negative polynomials (see e.g.\ Szeg\"o \cite[Chapter 1.21]{Szego}, this can be expressed as a sums of square condition, hence as a semidefinite program.
Then, when $d$ varies,  one obtains a sequence of semidefinite programs 
approaching $\vartheta_1(S^{n-1}, \Omega)$ from above.

\subsubsection{Case (ii): $\Omega=\{(x,y) : d_{\theta}(x,y)=\theta\}$.} 
The set $\Omega$ is again invariant under $O_n(\R)$ but now we deal
with \eqref{theta2}, which becomes:
\begin{equation}\label{Avoid theta}
\begin{array}{ll}
\vartheta_2(S^{n-1}, \Omega) =\displaystyle \sup\big\{\  f_0 :
\ f_k\geq 0,  & \displaystyle \sum_{k\geq 0} f_k =1,\\
&\displaystyle  \sum_{k\geq 0} f_k P_k^n(s) =0  \big\}
\end{array}
\end{equation}
where $s =\cos(\theta)$.
This linear program has infinitely many variables but only two
constraints. Its optimal value turns to be easy to determine
(we refer to Bachoc et. al. \cite{BachocNebeOliveiraVallentin} for a proof). 

\begin{mytheorem}
 Let $m(s)$ be the minimum of $P_k^n(s)$ for
  $k=0,1,2,\dots$. Then 
\begin{equation*}
\vartheta_2(S^{n-1}, \Omega)=\frac{m(s)}{m(s)-1}.
\end{equation*}
\end{mytheorem}

\subsubsection{Case (iii): $\Omega=\{(x,y) : d_{\theta}(x,y)\in
]0,\theta_{\min}[\text{ or }(x,y)\notin \Scap(e,\phi)^2\}$.} This set
is invariant by the smaller group $G=\Stab(e,O_n(\R))$. Like in case
(i), the program $\vartheta_1$ must be considered and in this program $F$ can be
assumed to be $G$-invariant. 

From Stone-Weierstrass theorem (see e.g. Conway \cite[Theorem 8.1]{Conway}), the elements of
$\CC((S^{n-1})^2)$ can be uniformly approximated by those of $V_d^{(2)}$. 
In addition, one can prove that the 
elements of $\CC((S^{n-1})^2)_{\succeq 0}$ can be uniformly approximated
by  positive definite functions belonging to
$V_d^{(2)}$.  We introduce:
\begin{equation}\label{theta1d}
\begin{array}{ll}
\vartheta_1^{(d)}(S^{n-1}, \Omega) =\inf\big\{\  t :  &F\in
\big(V_d^{(2)}\big)_{\succeq 0},\\
 & F(x,x)\leq t-1, \\
& F(x,y)\leq -1  \text{ for all }(x,y) \in \Omega^c\ \big\}.
\end{array}
\end{equation}
So, $\vartheta_1(S^{n-1}, \Omega)\leq
\vartheta_1^{(d)}(S^{n-1}, \Omega)$,  and
the limiting value of $\vartheta_1^{(d)}(S^{n-1}, \Omega)$ when $d$
goes to infinity equals $\vartheta_1(S^{n-1}, \Omega)$. 
Since $\Omega$ and $V_d$ are invariant by $G$, one can moreover assume
that $F$ is $G$-invariant. From
Theorems \ref{Th3} and \ref{Th dec2}, we have an expression for $F$:
\begin{equation*}
F(x,y)=\sum_{k=0}^d \langle F_k, Y_k(u,v,t)\rangle, \quad F_k\succeq 0,
\end{equation*}
where the matrices $F_k$ are of size $d-k+1$. Replacing in $\vartheta_1^{(d)}$ leads to:
\begin{equation*}
\begin{array}{lll}
\vartheta_1^{(d)}(S^{n-1},\Omega)=\inf\big\{  t : &F_k \succeq 0, &\\
& \displaystyle \sum_{k= 0}^d \langle  F_k, Y_k(u,u,1) \rangle\leq t\quad  &u\in [s',1],\\
& \displaystyle \sum_{k= 0}^d \langle  F_k, Y_k(u, v,t))\rangle \leq -1
\quad &s'\leq u\leq v\leq 1\\
&& -1\leq t\leq s\big\}
\end{array} 
\end{equation*}
with $s =\cos(\theta)$ and $s' =\cos(\phi)$. The left hand sides of
the inequalities are polynomials in three variables. Again, these
constraints can be relaxed using sums of squares in order to boil down
to true semidefinite programs. We refer to Bachoc, Vallentin \cite{BachocVallentin2} for the details,
and for numerical computations of upper bounds for codes in spherical
caps
with given minimal angular distance.

\subsection{Further applications}\label{Strengthening}
In Bachoc, Vallentin \cite{BachocVallentin1} it is shown how Delsarte
linear programming bound \cite{Delsarte1} can be improved 
with semidefinite constraints arising from the matrices $Y_k^n$ \eqref{Pdf
  2}. The idea is very much the same as for the Hamming space given in Schrijver 
\cite{Schrijver2} and explained in Section \ref{Block codes} : instead of considering constraints on pairs of
points only, one exploits constraints on triples of points. More
precisely,
if $S_k^n(u,v,t)$ denotes the symmetrization of $Y_k^n(u,v,t)$ in the
variables
$(u,v,t)$, then the following semidefinite property holds for all
spherical code $C$:
\begin{equation}\label{Skn}
\sum_{(x,y,z)\in C^3} S_k^n(x\cdot  y, y\cdot z, z\cdot x)\succeq 0.
\end{equation}
From \eqref{Skn}, it is possible to define a semidefinite program
whose optimal value upper bounds the number of elements of a code with given
minimal angular distance. In Bachoc, Vallentin \cite{BachocVallentin1}, Mittelmann, Vallentin \cite{MittelmannVallentin}, new upper
bounds for the kissing number have been obtained 
for the dimensions $n\leq 24$ with this method.
We give next a simplified version of the
semidefinite program used in \cite{BachocVallentin1}. Another useful version is given in Bachoc, Vallentin \cite{BachocVallentin4} for proving that the maximal angular distance of $10$ points on $S^3$ is $\cos(1/6)$. 

\begin{mytheorem}
The optimal value of the semidefinite program:
\begin{equation}\label{SDP}
\begin{array}{lll}
\inf\big\{\  1+\langle F_0, J_{d+1}\rangle : & F_k\succeq 0 & \\
& \displaystyle \sum_{k=0}^d \langle F_k, S_k^n (u,u,1)\rangle \leq
-\frac{1}{3},\;\;&-1\leq u\leq s\\
& \displaystyle \sum_{k=0}^d \langle F_k, S_k^n (u,v,t)\rangle \leq
0,&-1\leq u,v,t\leq s\ \}\\
\end{array}
\end{equation}
is an upper bound for the  number $A(S^{n-1}, \Omega, \mu_c)$ where
$\Omega=\{(x,y)\in (S^{n-1})^2 : s< x\cdot y <1\}$, i.e. for the maximal
number
of elements of a spherical code with minimal angular distance at least
equal to $\theta_{\min}=\arccos(s)$.
\end{mytheorem}

\begin{proof}
Let $(F_0,\dots,F_k)$ a feasible solution of \eqref{SDP}. Let 
\begin{equation*}
F(x,y,z) =\sum_{k=0}^d \langle F_k, S_k^n (x\cdot y, y\cdot z,z\cdot x)\rangle.
\end{equation*}
If $C$ is a spherical code, we consider $\Sigma =\sum_{(x,y,z)\in C^3}
F(x,y,z)$.
We have:
\begin{equation*}
0\leq \Sigma =\sum_{x\in C}
F(x,x,x)+\sum_{|\{x,y,z\}|=2}F(x,x,y) +\sum_{|\{x,y,z\}|=3}F(x,y,z)
\end{equation*}
where the inequality holds because of \eqref{Skn}. 
Then, taking $S_0^n(1,1,1)=J_{d+1}$ and $S_k^n(1,1,1)=0$
for $k\geq 1$ into account, we have $F(x,x,x)=\langle F_0, J_{d+1}\rangle$. 
If moreover  $d_{\theta}(C)\geq \theta_{\min}$, we can apply the
constraint inequalities of the program to the second and third terms
of the right hand side. We obtain:
\begin{equation*}
0\leq \Sigma \leq \langle F_0, J_{d+1}\rangle |C|-|C|(|C|-1)
\end{equation*}
leading to the inequality $|C|\leq 1+\langle F_0, J_{d+1}\rangle$.
\smartqed\qed
\end{proof}

\section{Sums of squares}
\label{Sums of squares}

A fundamental task in polynomial optimization and in real algebraic
geometry is to decide and certify whether a polynomial with real
coefficients in $n$ indeterminates can be written as a \emph{sum of
  squares}: \textit{Given $p \in \R[x_1, \ldots, x_n]$ do there exist polynomials $q_1,
\ldots, q_m \in \R[x_1, \ldots, x_n]$ so that
\begin{equation*}
p = q_1^2 + q_2^2 + \cdots + q_m^2 \; ?
\end{equation*}
}

This problem can be reformulated as a semidefinite feasibility problem: Let $z$ be a vector containing a
basis of $\R[x_1, \ldots, x_n]_{\leq d}$ the space of polynomials of
degree at most $d$. For example, let $z$ be the vector containing the
monomial basis
\begin{equation*}
z = (1, x_1, x_2, \ldots, x_n, x_1^2, x_1x_2, x_2^2, \ldots, x_n^d) 
\end{equation*}
which has length $\binom{n+d}{d}$. A polynomial $p \in \R[x_1, \ldots, x_n]$ of degree $2d$ is a sum of
square if and only if there is a positive semidefinite matrix $X$ of
size $\binom{n+d}{d} \times \binom{n+d}{d}$ so that the $\binom{n+2d}{2d}$ linear --- linear in the entries of $X$ --- equations
\begin{equation*}
p(x_1, \ldots, x_n)  = z^T X z
\end{equation*}
hold.

This semidefinite feasibility problem can be simplified if the
polynomial $p$ has symmetries. The method has been worked out by
Gatermann and Parrilo in \cite{GatermannParrilo}. In this section we give the main
ideas of the method. For details and further we refer to the
original article.

\subsection{Basics from invariant theory}

We start by explaining what we mean that a polynomial has
symmetries. Again to simplify the presentation of the theory we
consider the complex case only. 

Let $G$ be a finite group acting on $\C^n$. This group
action induces a group action on the polynomial ring $\C[x_1, \ldots,
x_n]$ by
\begin{equation*}
(gp)(x_1, \ldots, x_n) = p(g^{-1}(x_1, \ldots, x_n)),
\end{equation*}
and we say that a polynomial $p$ is $G$-invariant if $gp = p$ for all
$g \in G$. The set $\C[x_1, \ldots, x_n]^G$
of all $G$-invariant polynomials is a ring, the \emph{invariant
  ring}. By Hilbert's finiteness theorem it is generated by finitely
many $G$-invariant polynomials. Even more is true: Since the invariant
ring has the Cohen-Macaulay property it admits a Hironaka
decomposition: There are $G$-invariant polynomials $\eta_i$,
$\theta_j \in \C[x_1, \ldots, x_n]$ so that
\begin{equation}
\C[x_1, \ldots, x_n]^G = \bigoplus_{i=1}^r \eta_i \C[\theta_1, \ldots, \theta_s],
\end{equation}
hence, every invariant polynomial can be \emph{uniquely} written as a
polynomial in the polynomials $\eta_i$ and $\theta_j$ where $\eta_i$
only occurs linearly.  We refer to Sturmfels \cite[Chapter 2.3]{Sturmfels} for the
definitions and proofs; we only need the existence of a Hironaka
decomposition here.

\subsection{Sums of squares with symmetries}

We consider the action of the finite group $G$ restricted to the $\binom{n+d}{d}$-dimensional vector space of complex polynomials of degree at most $d$. This defines a unitary representation
\begin{equation*}
\pi : G \to \Gl(\C[x_1, \ldots, x_n]_{\leq d}).
\end{equation*}
From now on, by using the monomial basis, we see $\pi(g)$ as a regular matrix in $\C^{\binom{n+d}{d} \times \binom{n+d}{d}}$.

\begin{myexample}
For instance, the matrix $g^{-1} = \left(\begin{smallmatrix} 1 & 2\\ 3 & 4 \end{smallmatrix}\right)$ acts on $\C^2$ and so on the polynomial $p = 1 + x_1 + x_2 + x_1^2 + x_1 x_2 + x_2^2 \in \C[x_1, x_2]_{\leq 2}$ by
\begin{equation*}
\begin{split}
(gp)(x_1,x_2) & = p(x_1+2x_2, 3x_1 + 4x_2)\\
& =  1 + (x_1 + 2x_2) + (3x_1 + 4x_2) + (x_1+2x_2)^2 \\
& \quad + (x_1 + 2x_2)(3x_1+4x_2) + (3x_1+4x_2)^2\\
& = 1 + (x_1 + 2x_2) + (3x_1 + 4x_2) + (x_1^2 + 4x_1x_2 + 4x_2^2)\\
& \quad + (3x_1^2 + 10x_1x_2 + 8x_2^2) + (9x_1^2 + 24x_1x_2 + 16x_2^2).
\end{split}
\end{equation*}
and so defines the matrix 
\begin{equation*}
\pi(g) = 
\begin{pmatrix}
1 & 0 & 0 & 0 & 0 & 0\\
0 & 1 & 3 & 0 & 0 & 0\\
0 & 2 & 4 & 0 & 0 & 0\\ 
0 & 0 & 0 & 1 & 3 & 9\\
0 & 0 & 0 & 4 & 10 & 24\\
0 & 0 & 0 & 4 & 8 & 16
\end{pmatrix}.
\end{equation*}
\end{myexample}

Let $p \in \R[x_1, \ldots, x_n]$ be a polynomial which is a sum of squares and which is $G$-invariant. Thus we have a positive semidefinite matrix $X \in \R^{\binom{n+d}{d} \times \binom{n+d}{d}}$ so that
\begin{equation*}
p(x_1, \ldots, x_n) = z^{\sf T} X z = z^* X z,
\end{equation*}
and for every $g \in G$ we have
\begin{equation*}
gp(x_1, \ldots, x_n) = (\pi(g)^*z)^* X (\pi(g)^* z) = z^* \pi(g) X \pi(g)^* z.
\end{equation*}
Hence, $X$ is $G$-invariant and lies in $\left(\C^{\binom{n+d}{d} \times \binom{n+d}{d}}\right)^G$, the commutant algebra of the matrix $*$-algebra spanned by the matrices $\pi(g)$ with $g \in G$. So by Theorem~\ref{thm:blockdiagonal} there are numbers $D$, $m_1, \ldots, m_D$ and a $*$-isomorphism
\begin{equation*}
\varphi : \left(\C^{\binom{n+d}{d} \times \binom{n+d}{d}}\right)^G \to \bigoplus_{k=1}^D \C^{m_k \times m_k}.
\end{equation*}
Hence, cf.\ Step 2 (second version) in Section~\ref{ssec:sdp invariant}, we can write the polynomial $p$ in the form
\begin{equation*}
p(x_1, \ldots, x_n) = z^* \left(\sum_{k=1}^D \sum_{u,v = 1}^{m_k} x_{k,uv} \varphi^{-1}(E_{k,uv})\right) z
\end{equation*}
with $D$ positive semidefinite matrices
\begin{equation*}
X_k = \left(x_{k,uv}\right)_{1 \leq u,v \leq m_k}, \quad k = 1, \ldots, D.
\end{equation*}
We define $D$ matrices $E_1, \ldots, E_D \in (\C[x_1, \ldots, x_n]^G)^{m_k \times m_k}$ with $G$-invariant polynomial entries by
\begin{equation*}
(E_k)_{uv} = \left(\varphi^{-1}(E_{k,uv}) \right)_{\prod_i x_i^{\alpha_i}, \prod_i x_i^{\beta_i}} \prod_i x_i^{\alpha_i + \beta_i},
\end{equation*}
where we consider matrices in $\C^{\binom{n+d}{d} \times \binom{n+d}{d}}$ as matrices whose rows and columns are indexed by monomials $\prod_i x_i^{\alpha_i}$. Then, the polynomial $p$ has a representation of the form
\begin{equation*}
p(x_1, \ldots, x_n) = \sum_{k=1}^D \langle X_k, E_k \rangle.
\end{equation*}
Since the entries of $E_k$ are $G$-invariant polynomials we can use a Hironaka decomposition to represent them in terms of the invariants $\eta_i$ and $\theta_j$. We summarize our discussion in the following theorem.

\begin{mytheorem}
Let $p$ be a $G$-invariant polynomial of degree $2d$ which is a sum of squares. Then there are numbers $D$, $m_1, \ldots, m_D$ 
so that $p$ has a representation of the form
\begin{equation*}
p(x_1, \ldots, x_n) = \sum_{k=1}^D \langle X_k, E_k \rangle,
\end{equation*}
where $X_k \in \C^{m_k \times m_k}$ are positive semidefinite Hermitian matrices, and
where
\begin{equation*}
E_k \in \left(\bigoplus_{i=1}^r\eta_i\C[\theta_1, \ldots, \theta_s]\right)^{m_k \times m_k}
\end{equation*}
are matrices whose entries are polynomials (determined by a Hironaka decomposition of $\C[x_1, \ldots, x_n]^G$ and by the $*$-isomorphism $\varphi$).
\end{mytheorem}

\section{More applications}
\label{More applications}

In the last years many results were obtained for semidefinite programs which are symmetric. This was done for a variety of problems and applications. In this final section we want to give a brief, and definitely not complete, guide to the extensive and growing literature.

\subsection{Interior point algorithms}

Kanno, Ohsaki, Murota, Katoh \cite{KannaOhsakiMurotaKatoh} consider structural properties of search directions in primal-dual interior-point methods for solving invariant semidefinite programs and apply this to truss optimization problems.  de Klerk, Pasechnik \cite{deKlerkPasechnik1} show how the dual scaling method can be implemented to exploit the particular data structure where the data matrices come from a low-dimensional matrix algebra.

\subsection{Combinatorial optimization}

Using symmetry in semidefinite programs has been used in combinatorial optimization for a variety of problems: quadratic assignment problem (de Klerk, Sotirov \cite{deKlerkSotirov}), travelling salesman problem (de Klerk, Pasechnik, Sotirov \cite{deKlerkPasechnikSotirov}), graph coloring (Gvozdenovi\'c, Laurent \cite{GvozdenovicLaurent1}, \cite{GvozdenovicLaurent2}, Gvozdenovi\'c \cite{Gvozdenovic}), Lov\'asz theta number (de Klerk, Newman, Pasechnik,
  Sotirov \cite{deKlerkNewmanPasechnikSotirov}).

\subsection{Polynomial optimization}

Jansson, Lasserre, Riener, Theobald \cite{JanssonLasserreRienerTheobald} work out how constrained polynomial optimization problems behave which are invariant under the action of the symmetric group or the cyclic group.  Among many other things, Cimpri\u{c}, Kuhlmann, Scheiderer \cite{CimpricKuhlmannScheiderer} extend the discussion of Gatermann, Parrilo \cite{GatermannParrilo} from finite groups to compact groups. Cimpri\u{c} \cite{Cimpric} transfers the method to compute minima of the spectra of differential operators. In \cite{Bosse} Bosse constructs symmetric polynomials which are non-negative but not sums of squares.

\subsection{Low distortion geometric embedding problems}

Linial, Magen, Naor \cite{LinialMagenNaor} give lower bounds for low distortion embedding of graphs into Euclidean space depending on the girth. Vallentin \cite{Vallentin1} finds explicit optimal low distortion embeddings for several families of distance regular graphs. Both papers construct feasible solutions of semidefinite programs by symmetry reduction and by using the theory of orthogonal polynomials.

\subsection{Miscellaneous}

Bai, de Klerk, Pasechnik, Sotirov \cite{BaiKlerkPasechnikSotirov} exploit symmetry in truss topology optimization and Boyd, Diaconis, Parrilo, Xiao \cite{BoydDiaconisParriloXiao} in the analysis of fast mixing Markov chains on graphs.

\subsection{Software}

Pasechnik, Kini \cite{PasechnikKini} develop a software package for the computer algebra system {\tt GAP} for computing with the regular $*$-representation for matrix $*$-algebras coming from coherent configurations.

\subsection{Surveys and lecture notes}

Several surveys and lecture notes on symmetry in semidefinite programs with different aims were written in the last years. The lecture notes of Bachoc \cite{Bachoc2} especially discuss applications in coding theory and extend those of Vallentin \cite{Vallentin3} which focuses on aspects from harmonic analysis. The survey \cite{deKlerk} of de Klerk discusses next to symmetry also the exploitation of other structural properties of semidefinite programs like low rank or sparsity.

\section*{Acknowledgements}

We thank the referee for the helpful suggestions and comments.
The fourth author was supported by Vidi grant 639.032.917 from the Netherlands Organization for Scientific Research (NWO).


\begin{thebibliography}{[15]}

\bibitem{AndrewsAskeyRoy} 
G.E.~Andrews, R.~Askey, R.~Roy,
{\em Special functions}, 
Cambridge University Press, Cambridge, 1999.

\bibitem{Askey} 
R. Askey, 
{\em Orthogonal polynomials and special functions}, 
Society for Industrial and Applied Mathematics, Philadelphia, 1975.

\bibitem{Astola}
J. Astola, 
{\em The Lee-scheme and bounds for Lee-codes}, 
Cybernet. Systems {\bf 13} (1982) 331--343.

\bibitem{Bachoc1}
C. Bachoc,
{\em Linear programming bounds for codes in Grassmannian spaces},
IEEE Trans. Inf. Th.  {\bf 52} (2006), 2111--2125.

\bibitem{Bachoc2}
C. Bachoc,
{\em Semidefinite programming, harmonic analysis and coding theory},
arXiv:0909.4767 [cs.IT]

\bibitem{BachocVallentin1}
C. Bachoc, F. Vallentin,
{\em New upper bounds for kissing numbers from
          semidefinite programming},
J. Amer. Math. Soc. {\bf 21} (2008), 909--924.

\bibitem{BachocVallentin2}
C. Bachoc, F. Vallentin,
{\em Semidefinite programming, multivariate orthogonal polynomials, and codes in spherical caps}, 
special issue in the honor of Eichii Bannai, Europ. J.  Comb. {\bf  30}
(2009), 625--637.

\bibitem{BachocVallentin3}
C. Bachoc, F. Vallentin,
{\em More semidefinite programming bounds (extended abstract)}, pages 129--132 in Proceedings ``DMHF 2007: COE Conference on the Development of Dynamic Mathematics with High Functionality'', October 2007, Fukuoka, Japan.

\bibitem{BachocVallentin4}
C. Bachoc, F. Vallentin,
{\em Optimality and uniqueness of the (4,10,1/6) spherical code},
 J. Comb. Theory Ser. A {\bf 116} (2009), 195--204.

\bibitem{BachocZemor}
C. Bachoc, G. Z\'emor, 
{\em Bounds for binary codes relative to pseudo-distances of $k$ points},
Adv. Math. Commun. {\bf 4} (2010), 547--565.

\bibitem{BachocNebeOliveiraVallentin}
C. Bachoc, G. Nebe, F.M. de Oliveira Filho, F. Vallentin,
{\em Lower bounds for measurable chromatic numbers},
Geom. Funct. Anal. {\bf 19} (2009), 645--661.

\bibitem{BaiKlerkPasechnikSotirov} 
Y. Bai, E. de Klerk, D.V. Pasechnik, R. Sotirov,
{\em Exploiting group symmetry in truss topology optimization},
Optimization and Engineering {\bf 10} (2009), 331--349.

\bibitem{BannaiIto}
E. Bannai, T. Ito,
{\em Algebraic combinatorics. I.},
The Benjamin/Cummings Publishing Co. Inc., Menlo Park, CA, 1984.

\bibitem{Barg}
A. Barg, P. Purkayastha,
{\em Bounds on ordered codes and orthogonal arrays}, 
Moscow Math. Journal {\bf 9},  2009, 211--243.

\bibitem{Barvinok} 
A. Barvinok, 
{\em A course in convexity}, 
Graduate Studies in Mathematics {\bf  54}, 
American Mathematical Society, 2002.

\bibitem{Berger} 
M.~Berger, 
{\em A Panoramic View of Riemannian Geometry}, 
Springer-Verlag, 2003.

\bibitem{Bochner} 
S. Bochner,
{\em Hilbert distances and positive definite functions},
Ann. of Math. {\bf 42} (1941), 647--656.

\bibitem{Bosse} 
H. Bosse,
 {\em Symmetric, positive polynomials, which are not sums of squares}, 
 Series: CWI. Probability, Networks and Algorithms [PNA], Nr. E0706, 2007.

\bibitem{BoydDiaconisParriloXiao}
S. Boyd, P. Diaconis, P.A. Parrilo, L. Xiao,
{\em Symmetry analysis of reversible Markov chains},
Internet Mathematics {\bf 2} (2005).

\bibitem{BrouwerCohenNeumaier} 
A.E. Brouwer, A.M. Cohen and A. Neumaier, 
\emph{Distance-regular graphs}, 
Springer-Verlag, Berlin, 1989.

\bibitem{Bump} 
D. Bump, 
{\em Lie Groups}, 
Graduate Text in Mathematics {\bf 225}, Springer-Verlag, 2004.

\bibitem{Cameron}
P. J. Cameron,
{\em Coherent configurations, association schemes and permutation groups},
pages 55--71 in {\sl Groups, combinatorics \& geometry (Durham, 2001)},  World Sci. Publishing, River Edge, NJ, 2003.

\bibitem{Cimpric}
J. Cimpri\u{c},
{\em Estimating lowest eigenvalues of symmetric polynomial differential operators by semidefinite programming},
to appear in J. Math. Anal. Appl.

\bibitem{CimpricKuhlmannScheiderer}
J. Cimpri\u{c}, S. Kuhlmann, C. Scheiderer,
{\em Sums of squares and moment problems in equivariant situations},
Trans. Amer. Math. Soc.  {\bf 361}  (2009), 735--765. 

\bibitem{Conway} 
J.B. Conway, 
{\em A course in functional analysis},
Graduate Text in Mathematics {\bf 96}, Springer-Verlag, 2007.

\bibitem{ConwaySloane} 
J.H. Conway, N.J.A. Sloane, 
{\em Sphere Packings, Lattices and Groups}, 
third edition, Springer-Verlag, New York, 1999.

\bibitem{CreignouDiet} 
J. Creignou, H. Diet,
{\em Linear programming bounds for unitary codes}, 
Adv. Math. Commun. {\bf 4} (2010), 323--344.

\bibitem{Davidson}
K.R. Davidson,
{\em C*-Algebras by Example}.

\bibitem{Delsarte1} 
P. Delsarte, 
{\em An algebraic approach to the association schemes of coding
  theory},
Philips Res. Rep. Suppl. (1973), vi+97.

\bibitem{DelsarteGoethals} 
P. Delsarte, J. M. Goethals,
{\em Alternating bilinear forms over $GF(q)$},
J. Comb. Th. A {\bf 19} (1975), 26--50.

\bibitem{Delsarte2} 
P. Delsarte, 
{\em Hahn polynomials, discrete harmonics and $t$-designs},
SIAM J. Appl. Math. {\bf 34} (1978), 157--166.

\bibitem{Delsarte3} 
P. Delsarte, 
{\em Bilinear forms over a finite field, with applications to coding theory},
J. Comb. Th. A {\bf 25} (1978), 226-241.

\bibitem{DelsarteGoethalsSeidel} 
P. Delsarte, J.M. Goethals, J.J. Seidel,
{\em Spherical codes and designs},
Geom. Dedicata {\bf 6} (1977), 363--388.

\bibitem{DobredeKlerkPasechnik}
E. de Klerk, C. Dobre, D.V. Pasechnik, 
{\em Numerical block diagonalization of matrix *-algebras with application to semidefinite programming},
preprint, 2009.

\bibitem{Duffin} 
R.J. Duffin, 
{\em Infinite Programs}, 
in: Linear inequalities and related systems, (H.W. Kuhn, A.W. Tucker eds.), 
Princeton Univ. Press, 1956, 157--170.

\bibitem{Dunkl}
C.F. Dunkl, 
{\em A Krawtchouk polynomial addition theorem and wreath product of symmetric groups}, 
Indiana Univ. Math. J. {\bf 25} (1976), 335--358.

\bibitem{EricsonZinoviev} 
T. Ericson, V. Zinoviev,
{\em Codes on Euclidean spheres},
North-Holland, 2001.

\bibitem{GatermannParrilo}
K. Gatermann, P.A. Parrilo,
{\em Symmetry groups, semidefinite programs, and sums of squares},
J. Pure Appl. Alg. \textbf{192} (2004), 95--128.

\bibitem{GijswijtThesis}
D.C. Gijswijt,
{\em Matrix Algebras and Semidefinite Programming Techniques for Codes},
PhD thesis, University of Amsterdam,  2005. 

\bibitem{Gijswijt}
D.C. Gijswijt,
{\em Block diagonalization for algebra's associated with block codes},
arXiv:0910.4515 [math.OC]

\bibitem{GijswijtMittelmannSchrijver} 
D.C. Gijswijt, H.D. Mittelmann, A. Schrijver,
\emph{Semidefinite code bounds based on quadruple distances}, 
arXiv.math:1005.4959 [math.CO]

\bibitem{GijswijtSchrijverTanaka}
D.C. Gijswijt, A. Schrijver, H. Tanaka, 
{\em New upper bounds for nonbinary codes based on the Terwilliger algebra and semidefinite programming}, 
J. Comb. Theory Ser. A {\bf 113} (2006) 1719--1731. 

\bibitem{GoemansWilliamson}
M.X. Goemans, D.P. Williamson,
{\em Approximation algorithms for MAX-3-CUT and other problems via complex semidefinite programming},
J. Comput. System Sci. {\bf 68} (2004)  442--470.

\bibitem{Gvozdenovic} 
N. Gvozdenovi\'c, 
{\em Approximating the
    stability number and the chromatic number of a graph via
    semidefinite programming}, 
PhD Thesis, University of Amsterdam,
  2008.

\bibitem{GvozdenovicLaurent1} 
N. Gvozdenovi\'c, M. Laurent, 
{\em The operator $\Psi$ for the chromatic number of a graph},
 SIAM J. Optim. \textbf{19} (2008), 572--591. 

\bibitem{GvozdenovicLaurent2} 
N. Gvozdenovi\'c, M. Laurent, 
{\em Computing semidefinite programming lower bounds for the
    (fractional) chromatic number via block-diagonalization},
SIAM J. Optim. \textbf{19} (2008), 592-615.

\bibitem{GvozdenovicLaurentVallentin} 
N. Gvozdenovi\'c, M. Laurent, F. Vallentin, 
{\em Block-diagonal semidefinite programming hierarchies for 0/1 programming},
Oper. Res. Lett. {\bf 37} (2009), 27--31.

\bibitem{HornJohnson}
R. A. Horn, C. R. Johnson,
{\em Matrix analysis},
Cambridge University Press, Cambridge, 1990. 

\bibitem{JamesConstantine} 
A.T. James, A.G. Constantine,
{\em Generalized Jacobi polynomials as spherical functions of the Grassmann
manifold}, 
Proc. London Math. Soc. {\bf 29}  (1974), 174--192.

\bibitem{JanssonLasserreRienerTheobald} 
L. Jansson, J.B. Lasserre, C. Riener, T. Theobald,
{\em Exploiting symmetries in SDP-relaxations for polynomial optimization}, 
Optimization Online, September 2006.

\bibitem{KabatianskyLevenshtein} G.A.~Kabatiansky, V.I.~Levenshtein,
{\em Bounds for packings on a sphere and in space},
Problems of Information Transmission {\bf 14} (1978), 1--17.

\bibitem{KannaOhsakiMurotaKatoh} 
Y. Kanno, M. Ohsaki, K. Murota, N. Katoh, 
{\em Group symmetry in interior-point methods for semidefinite program},
Optimization and Engineering \textbf{2} (2001) 293--320.

\bibitem{Kleitman}
D.J. Kleitman, 
{\em The crossing number of $K_{5,n}$}, 
J.  Comb. Theory Ser. B {\bf 9} (1970) 315--323.

\bibitem{deKlerk}
E. de Klerk,
{\em Exploiting special structure in semidefinite programming: A survey of theory and applications},
European Journal of Operational Research {\bf 201} (2010), 1--10.

\bibitem{deKlerkMaharry}
E. de Klerk, J. Maharry, D.V. Pasechnik, R.B. Richter, G. Salazar,
{\em Improved bounds for the crossing numbers of $K_{m,n}$ and $K_n$},
SIAM J. Disc. Math. {\bf 20} (2006) 189--202.

\bibitem{deKlerkNewmanPasechnikSotirov} 
E. de Klerk, M.W. Newman, D.V. Pasechnik, R. Sotirov, 
{\em On the Lovasz $\theta$-number of
    almost regular graphs with application to Erd\H{o}s-R\'enyi
    graphs}, 
European J. Combin. {\bf 31} (2009), 879--888.

\bibitem{deKlerkPasechnik1} 
E. de Klerk, D.V. Pasechnik, 
{\em Solving SDP's in non-commutative algebras part I: the dual-scaling algorithm},
Discussion paper from Tilburg University, Center for economic
research, 2005.

\bibitem{deKlerkPasechnik2}
E. de Klerk, D.V. Pasechnik, 
{\em A note on the stability number of an orthogonality graph}, 
European J. Combin. {\bf 28} (2007) 1971--1979.

\bibitem{deKlerkPasechnikSchrijver}
E. de Klerk, D.V. Pasechnik, A. Schrijver, 
{\em Reduction of symmetric semidefinite programs using the regular $*$-representation},
Math. Program., Ser. B {\bf 109} (2007), 613--624.

\bibitem{deKlerkPasechnikSotirov} 
E. de Klerk, D.V. Pasechnik, R. Sotirov, 
{\em On Semidefinite Programming Relaxations of the Travelling Salesman Problem}, 
Discussion paper from Tilburg University, Center for economic
research, 2007.

\bibitem{deKlerkSotirov} 
E. de Klerk, R. Sotirov, 
{\em Exploiting group symmetry in semidefinite programming relaxations of the quadratic assignment problem}, 
Math. Program. Ser. A {\bf 122} (2010), 225--246.

\bibitem{Knuth} 
D.E.~Knuth, 
{\em The sandwich theorem},
Electron. J. Combin. \textbf{1} (1994), 48 pp.

\bibitem{Lam} 
T.Y. Lam, 
{\em A first course in noncommutative rings},
Springer-Verlag, 1991.

\bibitem{Lasserre} 
J.B Lasserre, 
{\em An explicit exact SDP relaxation for nonlinear 0/1 programs.},
pages 293--302 in K. Aardal and A.M.H. Gerards, eds., \emph{Lecture Notes in Computer Science} \textbf{2081}, 2001. 

\bibitem{LaurentTriples}
M. Laurent, 
{\em Strengthened semidefinite programming bounds for codes},
Math. Program. Ser. B \textbf{109} (2007) 239--261. 

\bibitem{LaurentSurvey} 
M. Laurent, 
{\em Sums of squares, moment matrices and optimization over polynomials},
pages 157--270 in {\sl Emerging Applications of Algebraic Geometry}, Vol. 149 of IMA Volumes in Mathematics and its Applications, M. Putinar and S. Sullivant (eds.), Springer-Verlag, 2009

\bibitem{Levenshtein} 
V.I. Levenshtein,
{\em Universal bounds for codes and designs}, in
{\em Handbook of Coding Theory}, eds V. Pless and W. C. Huffmann,
Amsterdam: Elsevier, 1998, 499-648.

\bibitem{LinialMagenNaor}
N. Linial, A. Magen, A. Naor,
{\em Girth and Euclidean Distortion}, 
Geom. Funct. Anal. {\bf 12} (2002),  380--394.

\bibitem{Lovasz} 
L. Lov\'asz,
{\em On the Shannon capacity of a graph}, 
IEEE Trans. Inform. Theory {\bf 25} (1979), 1--5.

\bibitem{MacWilliamsSloane}
F.J. MacWilliams, N.J.A. Sloane, 
{\em The theory of error-correcting codes},
North-Holland Mathematical Library, Vol. 16. North-Holland Publishing Co., Amsterdam-New York-Oxford, 1977.

\bibitem{MartinStinson}
W.J. Martin, D.R. Stinson, 
{\em Association schemes for ordered orthogonal arrays and $(T,M,S)$-nets},
Canad. J. Math. {\bf 51} (1999), 326--346. 

\bibitem{McElieceRodemichRumsey}
R.J. McEliece, E.R. Rodemich and H.C. Rumsey, Jr,
{\em The Lov\'asz bound and some generalizations},
J. Combinatorics, Inform. Syst. Sci. {\bf 3} (1978), 134--152.

\bibitem{MittelmannVallentin} 
H.D. Mittelmann, F. Vallentin, 
{\em High accuracy semidefinite programming bounds for kissing numbers},
Experiment. Math. {\bf 19} (2010), 174-178.

\bibitem{MurotaKannoKojimaKojima}
K. Murota, Y. Kanno, M. Kojima, S. Kojima, 
{\em A numerical algorithm for block-diagonal decomposition of matrix *-algebras, Part I: proposed approach and application to semidefinite programming}, 
to appear in Japan Journal of Industrial and Applied Mathematics.

\bibitem{Musin} 
O.R. Musin, 
{\em Multivariate positive definite functions on spheres}, 
arXiv:math/0701083 [math.MG]

\bibitem{Nemirovski} 
A.~Nemirovski, 
{\em Advances in convex optimization: Conic programming}, 
pages 413--444 in: {\sl Proceedings of
    International Congress of Mathematicians, Madrid, August 22-30,
    2006, Volume 1} (M.~Sanz-Sol, J.~Soria, J.L.~Varona, J.~Verdera,
  Eds.), European Mathematical Society Publishing House, 2007.

\bibitem{OdlyzkoSloane}
A.M. Odlyzko, N.J.A. Sloane,
{\em New bounds on the number of unit spheres that can touch a unit
  sphere in n dimensions}, 
J. Comb. Theory Ser. A {\bf 26} (1979), 210--214.

\bibitem{Oliveira}
F.M. de Oliveira Filho,
{\em New Bounds for Geometric Packing and Coloring via Harmonic Analysis and Optimization},
PhD thesis, University of Amsterdam, 2009.

\bibitem{OliveiraVallentin}
F.M. de Oliveira Filho, F. Vallentin,
{\em Fourier analysis, linear programming, and densities of distance avoiding sets in $\R^n$},
J. Eur. Math. Soc. (JEMS) {\bf 12} (2010), 1417--1428.

\bibitem{PasechnikKini}
D.V. Pasechnik, K. Kini,
{\em A GAP package for computation with coherent configurations},
pages 69--72 in: {\sl Mathematical Software --- ICMS 2010} (K.~Fukuda, J.~van der Hoeven, M.~Joswig, N.~Takayama, Eds.)
LNCS 6327, 2010.

\bibitem{Roy1}
 A. Roy, A. J. Scott,
{\em Unitary designs and codes},
Des. Codes Cryptogr. {\bf 53} (2009), 13--31.

\bibitem{Roy2} 
A. Roy, 
{\em Bounds for codes and designs in complex subspaces},
arXiv:0806.2317 [math.CO]

\bibitem{Rudin} 
W. Rudin, 
{\em Real and Complex Analysis}, 
McGraw-Hill International Editions, 1987.

\bibitem{Schoenberg} 
I.J.~Schoenberg,
{\em Positive definite functions on spheres},
Duke Math. J. {\bf 9} (1942), 96--108.


\bibitem{Schrijver1}
A. Schrijver, 
{\em Association schemes and the Shannon capacity: Eberlein polynomials and the Erd\H os-Ko-Rado theorem}, Algebraic methods in graph theory Vol. I, II (Szeged, 1978), pp. 671--688,
\emph{Colloq. Math. Soc. J\'anos Bolyai} 25, North-Holland, Amsterdam-New York, 1981.

\bibitem{Schrijver2} 
A. Schrijver,
{\em New code upper bounds from the Terwilliger algebra and
  semidefinite programming}, 
IEEE Trans. Inform. Theory {\bf 51} (2005), 2859--2866. 

\bibitem{Soifer}
A. Soifer,
{\em The mathematical coloring book},
Springer-Verlag, 2008.

\bibitem{Stanton}
D. Stanton,
{\em Orthogonal polynomials and Chevalley groups}, in 
Special functions: group theoretical aspects and applications,
R.A. Askey, T.H. Koornwinder and W. Schempp (Eds.),
Reidel Publishing Compagny, 1984.

\bibitem{Sturm}
J.F.~Sturm,
{\em Using SeDuMi 1.02, A Matlab toolbox for optimization over symmetric cones},
Optimization Methods and Software {\bf 11} (1999), 625--653.

\bibitem{Sturmfels}
B.~Sturmfels,
{\em Algorithms in Invariant Theory},
Springer-Verlag, 1993. 

\bibitem{Szego} 
G. Szeg\"o, 
{\em Orthogonal polynomials}, 
American Mathematical Society, 1939.

\bibitem{Takesaki}
M. Takesaki, 
{\em Theory of Operator Algebras I.},
Encyclopaedia of Mathematical Sciences, 124. Operator Algebras and Non-commutative Geometry, 5. Springer-Verlag, Berlin, 2002.

\bibitem{TarnanenAaltonenGoethals} 
H. Tarnanen, M.  Aaltonen, J.-M.  Goethals,
{\em On the nonbinary Johnson scheme},  
European J. Combin.  {\bf 6}  (1985), 279--285. 

\bibitem{Tarnanen}
H. Tarnanen, 
{\em Upper bounds on permutation codes via linear programming}, 
European  J. Combin. {\bf 20} (1999) 101--114.

\bibitem{Vallentin1} 
F. Vallentin, 
{\em Optimal embeddings of distance transitive graphs into Euclidean spaces}, 
J. Comb. Theory Ser. B {\bf 98} (2008), 95--104.

\bibitem{Vallentin2} 
F. Vallentin, 
{\em Symmetry in semidefinite programs}, 
Linear Algebra and Appl. 430 (2009), 360--369. 

\bibitem{Vallentin3}
F. Vallentin,
{\em Lecture notes: Semidefinite programs and harmonic analysis},
arXiv:0809.2017 [math.OC]

\bibitem{VilenkinKlimyk} 
N.Ja.~Vilenkin, A.U.~Klimyk,
{\em Representation of Lie Groups and Special Functions, Volume 2},
Kluwer Academic Publishers, 1993.

\bibitem{Wang} 
H.G. Wang, 
{\em Two-point homogeneous spaces},
Ann. of Math. {\bf 55} (1952), 177--191.

\bibitem{Woodall}
D.R. Woodall, 
{\em Cyclic-order graphs and Zarankiewicz's crossing-number conjecture}, 
J. Graph Theory {\bf 17} (1993) 657--671.

\bibitem{Zarankiewicz}
K. Zarankiewicz, 
{\em On a problem of P. Tur\'an concerning graphs},
Fundamenta Mathematicae {\bf 41} (1954) 137--145.

\end{thebibliography}
\end{document}